\documentclass[onetabnum,final]{siamart220329}

\usepackage{lipsum}
\usepackage{booktabs}
\usepackage{amsfonts}
\usepackage{graphicx}
\usepackage{epstopdf}
\usepackage{algorithmic}
\ifpdf
  \DeclareGraphicsExtensions{.eps,.pdf,.png,.jpg}
\else
  \DeclareGraphicsExtensions{.eps}
\fi

\usepackage{amssymb}
\usepackage{amsmath}
\usepackage{amsfonts} 
\usepackage{mathtools}

\newsiamremark{remark}{Remark}
\newsiamremark{hypothesis}{Hypothesis}
\crefname{hypothesis}{Hypothesis}{Hypotheses}
\newsiamthm{claim}{Claim}

\headers{CutFEM for EMI}{Berre et al}

\title{Cut finite element discretizations of cell-by-cell EMI electrophysiology models\thanks{Submitted to pre-print server and journal (\today).
\funding{M.~E.~Rognes graciously acknowledges support and funding from the Research Council of Norway (RCN) via FRIPRO grant agreement \#324239 (EMIx)}}}

\author{
  Nanna Berre\thanks{Department of Mathematical Sciences, Norwegian University of Science and Technology, Trondheim, Norway (\email{nanna.berre@ntnu.no})}
  \and Marie E.~Rognes\thanks{Department for Numerical Analysis and Scientific Computing, Simula Research Laboratory, Oslo, Norway (\email{meg@simula.no})}
  \and André Massing\thanks{Department of Mathematical Sciences, Norwegian University of Science and Technology, Trondheim, Norway (\email{andre.massing@ntnu.no})}
}

\usepackage{amsopn}

\makeatletter
\newcommand*{\addFileDependency}[1]{
  \typeout{(#1)}
  \@addtofilelist{#1}
  \IfFileExists{#1}{}{\typeout{No file #1.}}
}
\makeatother

\usepackage{todonotes}

\newcommand{\andre}[1]{\textcolor{cyan}{AM: #1}}

\newcommand{\R}{\mathbb{R}}

\DeclarePairedDelimiter{\norm}{\lVert}{\rVert} 
\DeclarePairedDelimiter{\abs}{\lvert}{\rvert}  
\newcommand{\T}{\mathcal{T}}
\newcommand{\F}{\mathcal{F}}

\newcommand{\I}{\mathcal{I}}

\newcommand{\vertiii}[1]{{\left\vert\kern-0.25ex\left\vert\kern-0.25ex\left\vert #1 
    \right\vert\kern-0.25ex\right\vert\kern-0.25ex\right\vert}}
    
\newcommand{\dx}{\,\mathrm{d}x}
\newcommand{\ds}{\,\mathrm{d}s}

\usepackage{overpic}
\usepackage{subcaption}
\usepackage{siunitx}
\usepackage{mathtools}
\usepackage{rotating}
\usepackage{multirow}

\newcommand{\Hdiv}{$H(\mathrm{div})$}

\ifpdf
\hypersetup{
  pdftitle={CutFEM for EMI},
  pdfauthor={N.~Berre, M.~E.~Rognes, and A.~Massing}
}
\fi

\begin{document}

\maketitle

\begin{abstract}
The EMI (Extracellular-Membrane-Intracellular) model describes
electrical activity in excitable tissue, where the extracellular and
intracellular spaces and cellular membrane are explicitly represented.
The model couples a system of partial differential equations in the
intracellular and extracellular spaces with a system of ordinary
differential equations on the membrane. A key challenge for the EMI
model is the generation of high-quality meshes conforming to the
complex geometries of brain cells. To overcome this challenge, we propose
a novel cut finite element method (CutFEM) where the 
membrane geometry can be represented independently of a structured
and easy-to-generated background mesh for the remaining computational domain. 
    
Starting from a Godunov splitting scheme, the EMI model is split into separate PDE and ODE parts. The resulting PDE part is a non-standard elliptic interface problem, for which we
devise two different CutFEM formulations: one single-dimensional formulation with the intra/extracellular electrical potentials as unknowns, and a multi-dimensional formulation  that also introduces the electrical current over the membrane as an additional unknown leading to a penalized saddle point problem.  Both formulations are  augmented by suitably designed ghost penalties to ensure stability and convergence properties that are insensitive to how the membrane surface mesh cuts the background mesh. For the ODE part, we introduce a new unfitted discretization to solve the membrane bound ODEs on a membrane interface that is not aligned with the background mesh. Finally, we perform extensive numerical experiments to demonstrate that CutFEM is a promising approach to efficiently simulate electrical activity in geometrically resolved brain cells.

\end{abstract}

\begin{keywords}
  cut finite elements, electrophysiology, coupled ODE-PDEs, neuroscience
\end{keywords}

\begin{AMS}
  65M60, 65M85, 65N30, 65N85, 92C20
\end{AMS}

\section{Introduction}
Today, classical paradigms for modelling electrophysiology such as the
bidomain model~\cite{sundnes2006computational} are being challenged by
new mathematical frameworks that explicitly represent and resolve the
geometry of extracellular and intracellular
spaces~\cite{agudelo-toroComputationallyEfficientSimulation2013,
  tveito2017cell, mori2009numerical, ellingsrud2020finite,
  stinner2020mathematical, jaegerEfficientNumericalSolution2021,
  de2023boundary}. This evolution is jointly driven by impressive
advances in medical physics and technology, allowing for in-vivo
imaging of intracellular and extracellular dynamics as well as
microscopy at the subcellular level, and by advances in computational
resources, making the previously uncomputable computable. These
cell-based frameworks, referred to as
\emph{extracellular-membrane-intracellular} (EMI)
models~\cite{tveito2017cell} or \emph{cell-by-cell electrophysiology}
models~\cite{de2023boundary} or \emph{electrophysiology with internal
  boundaries}~\cite{mori2009numerical}, take the form of coupled
systems of ordinary and partial differential equations (ODEs,
PDEs). Typically, the PDEs are defined relative to disjoint domains of
dimension $d$ coupled over a topologically $d-1$-dimensional internal
interface, while the ODEs are defined over the lower-dimensional
manifold alone. These types of equations frequently appear in
electrophysiology: to model the electrical activity of heart cells
with polarized membrane
features~\cite{jaegerEfficientNumericalSolution2021} or that in
neurons, extracellular space and glial
cells~\cite{agudelo-toroComputationallyEfficientSimulation2013,
  ellingsrud2020finite}, while the PDE subproblems are also frequently
encountered in chemical and electrical engineering such as for
modelling thermal contact
resistance~\cite{belgacem2015compositemedia}.

In all of these cases, the geometry is chiefly defined by the
lower-dimensional surface, e.g.~the cell membranes or composite media
interfaces, frequently in the form of an image-derived STL
surface. The associated cellular domains may be
non-convex, narrow and tortuous, placing high demands on the
resolution required for conforming volumetric meshes and essentially
forcing the numerical resolution to be dictated by geometrical rather
than approximation considerations.  
Moreover, without suitable post-processing, the quality of the
generated STL surface meshes might be insufficient to generate 
high-quality 3D unstructured volume meshes required for accurate finite
element-based computations.

A natural alternative is to instead consider unfitted meshes and
discretizations such as cut finite element methods
(CutFEM)~\cite{burmanCutFEMDiscretizingGeometry2015}, finite cell
methods (FCM)~\cite{RankRuessKollmannsbergerEtAl2012}, shifted
boundary method (SBM)~\cite{AtallahCanutoScovazzi2021}, aggregated
unfitted finite element methods (AgFEM)~\cite{BadiaVerdugoMartin2018},
as well as unfitted discontinuous Galerkin methods
~\cite{BastianEngwer2009,gurkanStabilizedCutDiscontinuous2019} and
methods based on immersogeometric
analysis~\cite{SchillingerDedeScottEtAl2012},
see~\cite{VerhooselLarsonBadiaEtAl2023} for a recent state-of-the-art
review.  A key advantage of the unfitted finite element technology is
that allows to embed the image-derived, and potentially lower quality,
surface mesh into an easy-to-generate structured background mesh in a
non-conforming manner without scarifying the approximation power of
the underlying finite element spaces. As a result, the costly
generation of high-resolution and high-quality unstructured 3D volume
meshes for finite element based computations is completely
circumvented, potentially leading to more streamlined and
semi-automated simulation pipelines. Such image-based finite element
analyses of biological media has been successfully used
in~\cite{SchillingerRuess2014,VerhooselvanZwietenvanRietbergenEtAl2015,ClausKerfridenMoshfeghifarEtAl2021},
to predict the elastic responses of bone structures. However, the use
of image-based models is barely explored in the computational
neuroscience community, and only very recently, unfitted finite
element methods have been employed to simulate astrocytic metabolism
in realistic three-dimensional
geometries~\cite{farina2021cut,Farina2022}, as well to solve the
electroencephalography (EEG) forward
problem~\cite{NuessingWoltersBrinckEtAl2016,ErdbrueggerWesthoffHoeltershinkenEtAl2022}.

\begin{figure}
\captionsetup[subfigure]{}
\hfill
\begin{subfigure}[T]{0.45\textwidth}
    \includegraphics[width=\textwidth]{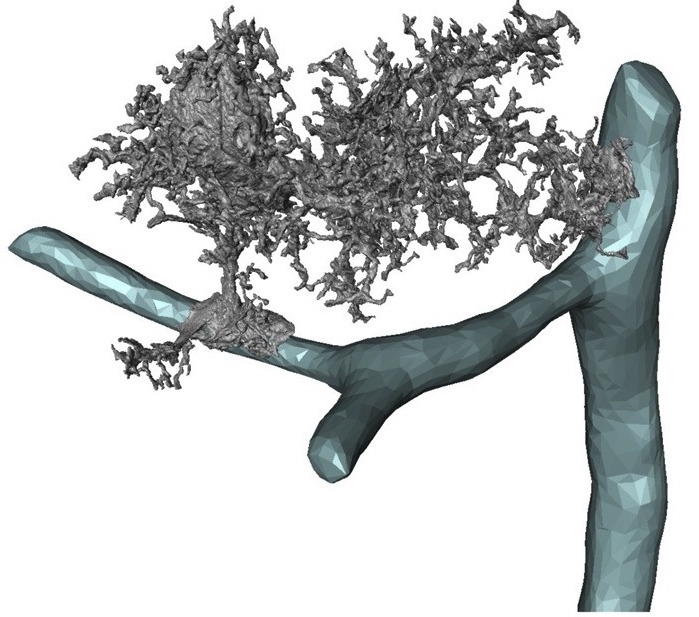}
\end{subfigure}
\hfill
\begin{subfigure}[T]{0.49\textwidth}
    \includegraphics[width=\textwidth]{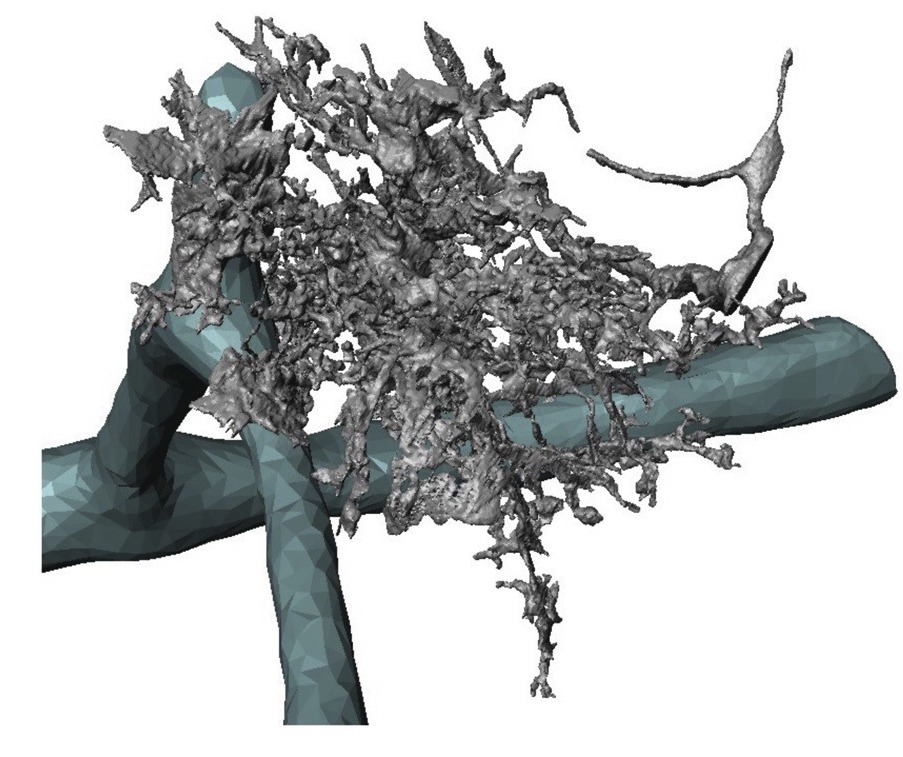}
\end{subfigure}
\caption{Example of 3d reconstruction of astrocytes morphology from \cite{CALI2019}, adapted(cropped) under the Creative Commons CC-BY license. }
\label{fig:cali}
\end{figure}

\subsection{New contributions and outline of the paper}
In all the aforementioned contributions, unfitted finite element
methods were devised to solve elliptic and parabolic problems with
boundary conditions posed on imaged-derived surfaces meshes. The
computational modeling of electric activity in geometrically resolved
excitable cells on the other hand requires to consider
\emph{multi-dimensional} surface-bulk problems which couples an
elliptic interface PDE defined on the extra/intracellular bulk domains
with a nonlinear system of ODEs solely posed on the membrane surface.
In the present work, we introduce two cut finite element methods for
the robust and accurate discretization of such multi-dimensional EMI
systems. Our method departs from an earlier developed Godunov
splitting scheme for the
model~\cite{tveito2017cell}, which decouples
the PDE part from the ODE part. For the PDE step, different weak
formulations can be derived~\cite{EMIchapter5} and we propose
CutFEM-based discretizations for two of them. First, we consider a
single-dimensional formulation with the intra/extracellular electrical
potentials as unknowns. This approach is inspired
by~\cite{jiUnfittedFiniteElement2017} who analyzed and successfully
applied such a discretization for approximating the temperature in
composite media with contact resistance. In the second,
multi-dimensional formulation, the electrical current over the
membrane is introduced as an additional unknown which formally results
in a generalized saddle point problem with a penalty-like term.
By augmenting the considered weak formulations with
suitably designed, so-called ghost
penalties~\cite{burmanCutFEMDiscretizingGeometry2015}, we ensure that
the resulting discretizations are geometrically robust so that their
stability and accuracy is not affected by how the membrane surface
mesh and the structured background mesh intersect.

In addition to the proposed CutFEM-based discretizations of elliptic problem
with a contact resistance interface condition,
we also introduce a CutFEM formulation of the ODE system posed on the unfitted membrane surface.
In contrast to fitting approaches,
solving the ODE system directly in the surface mesh nodes
makes it challenging from an implementation point of view
to interface with the PDE problem.
Moreover, the STL mesh can be of low quality with potentially large
variations in both the size and the aspect ratio of the individual triangles.
Our approach borrows ideas from the CutFEM discretization of surface PDEs~\cite{BurmanHansboLarsonEtAl2018}
and allows for an easy and seamless integration of the ODE and PDE step.
Solvers and simulation code for the present work are based on
Julia~\cite{bezansonJuliaFreshApproach2017} and the finite element
framework Gridap~\cite{badiaGridapExtensibleFinite2020}, and is openly
available~\cite{nanna_berre_2023_8068506}.
In summary, the main contributions of this paper are as follows.
\begin{itemize}
    \item
          We introduce a complete unfitted discretization approach for the
          EMI/cell-by-cell electrophysiology equations and demonstrate its
          robustness and applicability on idealized and image-based
          geometries.
    \item
          We propose different CutFEM discretization techniques for the
          Poisson interface problem, both single-dimensional and
          multi-dimensional formulations and provide detailed numerical
          evidence of their well-posedness, approximation properties, and
          geometrical robustness.
    \item
          We introduce an unfitted discretization technique for systems of
          nonlinear, spatially-dependent ODEs defined over the interface in
          the absence of a conforming interface representation.
\end{itemize}

This paper is structured as follows. We describe the cell-by-cell
electrophysiology (EMI) model in~\Cref{Section:model}, and outline a
classical operator splitting algorithm separating this system of
time-dependent, nonlinear, coupled equations into a set of ODEs and a
set of PDEs in the
brief~\Cref{section:splitting}. In~\Cref{section:spatial}, we
introduce two unfitted (CutFEM) formulations of the PDEs: a
single-dimensional formulation for the intracellular and extracellular
potential alone, and a multi-dimensional formulation for the
intracellular and extracellular potentials as well as the membrane
flux. We prove that both formulations are well-posed and discuss
unfitted stabilization operators. We turn to the system of ODEs
in~\Cref{sec:ode} and introduce a discretization technique for such
system of equations living on an interface, but without an explicit
discrete representation of the interface. We study the approximation
and solvability properties of the ODE-, PDE- and coupled
discretizations in~\Cref{sec:numerical}, before
concluding in~\Cref{sec:conclusion}.


\section{The cell-by-cell EMI electrophysiology model}
\label{Section:model}

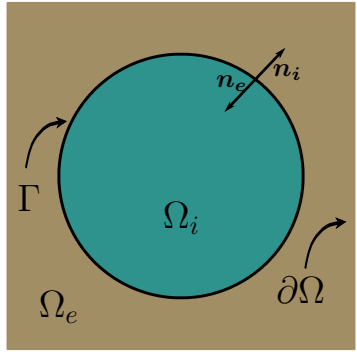
\begin{figure}
    \centering
    \begin{overpic}[scale = 0.11]{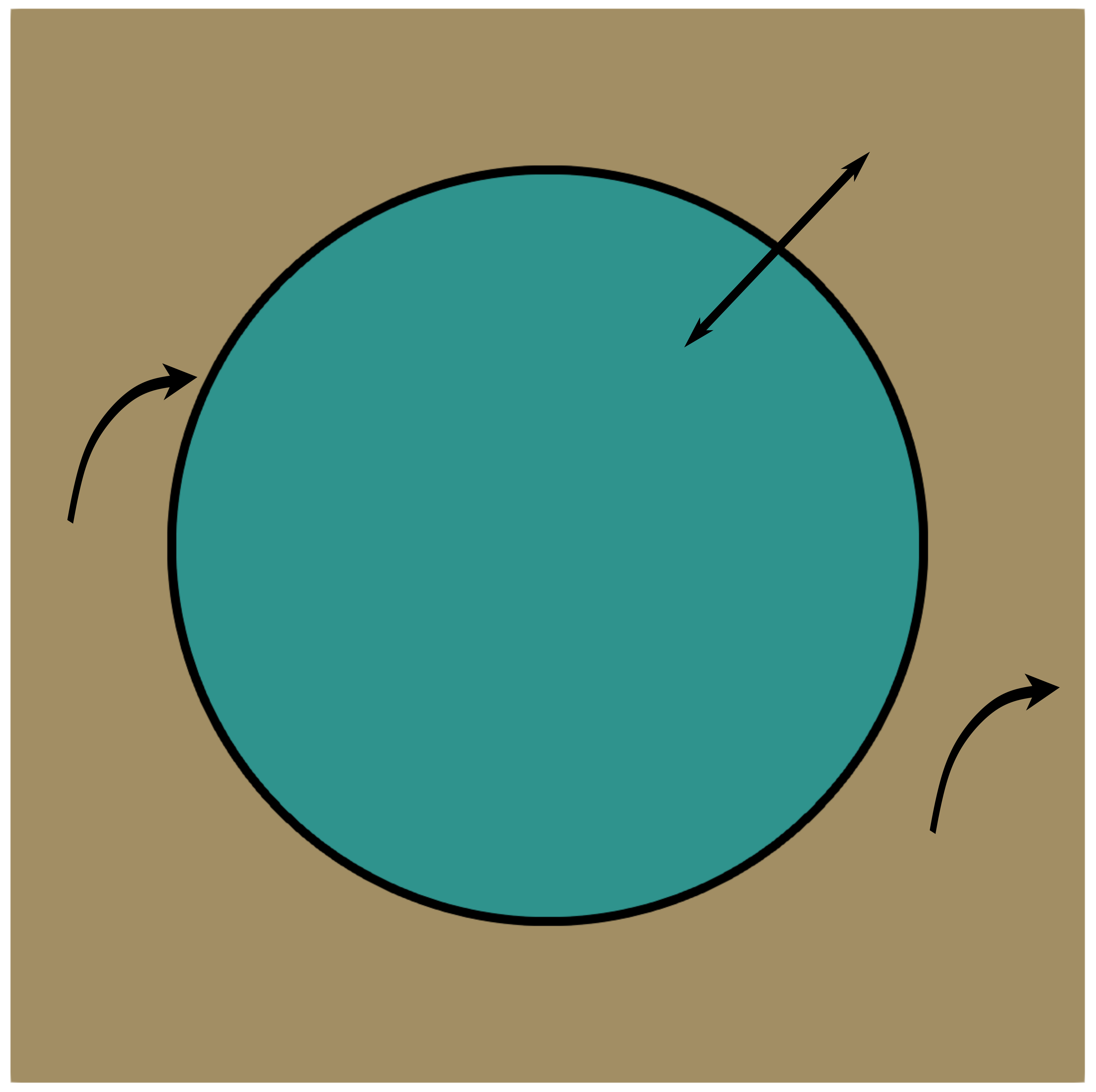}
        \put(76,78){\(\boldsymbol{n_i}\) }
        \put(60,75){\(\boldsymbol{n_e}\) }
        \put(45,35){\Large\(\Omega_i\)}
        \put(10,10){\Large\(\Omega_e\)}
        \put(4,40){\Large\(\Gamma\)}
        \put(77,14){\Large\(\partial\Omega\)}
    \end{overpic}
    \caption{Illustration of the domains in the EMI model.}
    \label{Figure:cell}
\end{figure}
The EMI
model~\cite{agudelo-toroComputationallyEfficientSimulation2013,tveito2017cell,EMIchapter1,
    de2023boundary} is a multi-dimensional multi-physics problem that
couples a linear Poisson-interface problem in the intra- and
extracellular domains to a system of nonlinear ODEs on the membrane
interface. More precisely, assume that the complete computational
domain $\Omega$ is separated into an extracellular domain $\Omega_e$
and an intracellular domain $\Omega_i$ with
\(\partial \Omega_i \cap \partial \Omega_e = \emptyset \), and
such that \(\Omega_i\) is strictly contained in \(\Omega \). The
membrane of the cell is defined as the intersection of the closures of
$\Omega_e$ and $\Omega_i$ and is denoted by $\Gamma$
(\cref{Figure:cell}). The EMI model describes the spatial-temporal
extracellular potential \(u_e\), intracellular potential \(u_i\), and
the transmembrane potential \(v\) governed by the following equations:
\begin{subequations}
    \label{emi}
    \begin{alignat}{2}
        -\nabla \cdot \left(\sigma_e \nabla u_e\right)  & = 0                                                      &  & \quad \quad \text{in }\Omega_e \label{Emi1},  \\
        -\nabla \cdot \left( \sigma_i \nabla u_i\right) & = 0                                                      &  & \quad \quad \text{in }\Omega_i , \label{Emi2} \\
        \sigma_e \nabla u_e \cdot \boldsymbol{n_e}      & = -\sigma_i \nabla u_i \cdot \boldsymbol{n_i} \equiv I_m &  & \quad \quad \text{on } \Gamma,\label{Emi3}    \\
        v                                               & = u_i - u_e                                              &  & \quad \quad \text{on } \Gamma,\label{Emi4}    \\
        C_m \partial_t v                                & = I_m - I_{\mathrm{ion}}(v,s)                            &  & \quad \quad \text{on } \Gamma,\label{Emi5}    \\
        \partial_t s                                    & = F(s,v)                                                 &  & \quad \quad \text{on }\Gamma,
    \end{alignat}
\end{subequations}
where the intracellular and extracellular conductivities are denoted by \(\sigma_i\) and \(\sigma_e\), and \(C_m\) is the membrane capacitance. Next, the ionic current density is given by \(I_{\text{ion}}\) and the total membrane current density by \(I_m\). Finally, \( \boldsymbol{n_i}\) denotes the normal pointing out from \(\Omega_i\), and \( \boldsymbol{n_e}=-\boldsymbol{n_i}\).

\begin{table}
    \begin{center}
        \caption{Parameter values for the Hodgkin-Huxley model}
        \label{table:hh_parameters}
        \begin{tabular}{  l c c  c c c }
            \toprule
            Parameter                     & Symbol                  & Value                 & Unit                                      & Reference                                    \\
            \midrule
            Membrane capacitance          & \(C_m\)                 & \(2 \cdot 10^{-5}\)   & \(\frac{\si{\nano F}}{\si{\micro m}^2}\)  & \cite{tveitoEvaluationAccuracyClassical2017} \\
            Na+ HH max conductivity       & \(\bar{g}_{Na}\)        & \(1.2 \cdot 10^{-3}\) & \(\frac{\si{\micro S}}{\si{\micro m}^2}\) & \cite{Hodgkin1952}                           \\
            K+ HH max conductivity        & \(\bar{g}_{K}\)         & \(3.6 \cdot 10^{-4}\) & \(\frac{\si{\micro S}}{\si{\micro m}^2}\) & \cite{Hodgkin1952}                           \\
            Stimulus                      & \( g_{\mathrm{stim}} \) & \(7 \cdot 10^{-3}\)   & \(\frac{\si{\micro S}}{\si{\micro m}^2}\) &                                              \\
            Leak conductivity             & \(\bar{g}_{L}\)         & 3 \(\cdot 10^{-6}\)   & \(\frac{\si{\micro S}}{\si{\micro m}^2}\) & \cite{Hodgkin1952}                           \\
            Na+ equilibrium potential     & \(E_{Na}\)              & 50                    & \(\si{mV}\)                               & \cite{Hodgkin1952}                           \\
            K+ equilibrium potential      & \(E_{K}\)               & -77                   & \(\si{mV}\)                               & \cite{Hodgkin1952}                           \\
            Leak equilibrium potential    & \(E_{L}\)               & -54.5                 & \(\si{mV}\)                               & \cite{Hodgkin1952}                           \\
            Initial membrane potential    & \(v^0\)                 & -67.7                 & \(\si{mV}\)                               &                                              \\
            Resting membrane potential    & \(v_{rest}\)            & -65                   & \(\si{mV}\)                               &                                              \\
            Intracellular conductivity    & \(\sigma_i\)            & 0.7                   & \(\frac{\si{\micro S}}{\si{\micro m}} \)  & \cite{tveitoEvaluationAccuracyClassical2017} \\
            Extracellular conductivity    & \(\sigma_e\)            & 0.3                   & \(\frac{\si{\micro S}}{\si{\micro m}}\)   & \cite{tveitoEvaluationAccuracyClassical2017} \\
            Initial HH gating value (Na+) & \(m^0\)                 & 0.0379                &                                           & \cite{Hodgkin1952}                           \\
            Initial HH gating value (Na+) & \(h^0\)                 & 0.688                 &                                           & \cite{Hodgkin1952}                           \\
            Initial HH gating value (K+)  & \(n^0\)                 & 0.276                 &                                           & \cite{Hodgkin1952}                           \\
            \bottomrule
        \end{tabular}
    \end{center}
\end{table}
The ionic current density \(I_{\text{ion}}\) governs the membrane potential and is subject to further modeling. In this work, the Hodgkin-Huxley model \cite{Hodgkin1952} will be used, where \(I_{\text{ion}}\) is given as follows
\begin{align}
    I_{\mathrm{ion}} & =\bar{g}_{Na}m^3h(v-E_{Na}) +\bar{g}_{K}n^4(v-E_{K}) + \bar{g}_{L}(v-E_L) - I_{\mathrm{app}},
\end{align}
with \(\bar{g}_{Na}\), \(\bar{g}_{K}\) and \(\bar{g}_{L}\) being the maximal conductivities related to sodium, potassium and a leak channel, respectively. The gating variables \(m\) and \(h\) are used to model the probability of gates being open to the passage of sodium ions, and the corresponding gating variable for potassium is denoted by \(n\). Equilibrium potentials, which indicate at which potentials sodium, potassium, and the leak channel are in equilibrium, are denoted by \(E_{Na}\), \(E_{K}\), and \(E_L\), respectively.
Importantly, these gating variables are modeled with the following ODEs
\begin{align}
    p_t = \alpha_p(v)(1-p)- \beta_p(v)p,
\end{align}
for \(p\in\{m,h,n\}\), where the following expressions are used for \(\alpha\) and \(\beta\)~\cite{Hodgkin1952},
\begin{alignat*}{3}
    \alpha_m & = 0.1\frac{25 - v_M}{\exp((25 - v_M)/10) - 1}, &  & \qquad
    \beta_m = 4\exp(-v_M/18),                                             \\
    \alpha_h & = 0.07\exp(-v_M/20),                           &  & \qquad
    \beta_h = \frac{1}{\exp((30-v_M)/10) + 1},                            \\
    \alpha_n & = 0.01\frac{10-v_M}{\exp((10-v_M)/10) - 1},    &  & \qquad
    \beta_n = 0.125\exp(-v_M/80).
\end{alignat*}
Here \(v_M = v - v_{\mathrm{rest}} \) with \(v_{\mathrm{rest}}\) being the resting potential. Parameter values for the Hodgkin-Huxley model are listed in~\Cref{table:hh_parameters}.

In addition, we impose an applied stimulus \(I_{\mathrm{app}}\)
\begin{align*}
    I_{\mathrm{app}} = g_{\mathrm{stim}} H(x) T(t)
\end{align*}
where
\begin{align*}
    H(x) = \begin{cases}
               1 & \text{ if } \boldsymbol{x} \in D, \\
               0 & \text{else},
           \end{cases}
    \qquad \quad
    T(t) = \begin{cases}
               1 & \text{ if } \boldsymbol{t} \in [t_1,t_2], \\
               0 & \text{ else},
           \end{cases}
\end{align*}
are the indicator functions of some domain \(D\subset \Omega\) and some time interval \([t_1,t_2]\).

\section{Operator splitting and time discretization of the EMI model}
\label{section:splitting}

The EMI model with Hodgkin-Huxley membrane dynamics consists of a PDE system coupled with an ODE system. We employ a standard operator splitting scheme to decouple the problem into smaller subproblems. Each subproblem can then be solved with solution methods dedicated to the specific subproblem. A classical approach is Godunov splitting~\cite{levequeFiniteVolumeMethods2002}, which has previously successfully been applied to the EMI model \cite{tveito2017cell,jaegerPropertiesCardiacConduction2019}, to decouple the membrane-confined ODE system from the PDE system.

To discretize the PDE step in time we apply an implicit Euler discretization and denote the size of the time step by \(\Delta t\).
For the ODE system, we discretize pointwise in time, using an explicit Euler step to pass from \(t^m\) to \(t^{m+1} = t^{m}+\Delta t\). The time-discretized splitting scheme is given in \cref{alg:time_disc}.
\begin{algorithm}
    \caption{Time-discretized splitting scheme for the EMI model}
    \label{alg:time_disc}
    \begin{algorithmic}
        \STATE{Given initial values \(v^0\), \(s^0\)}
        \FOR{timestep \(m=1:M\)}
        \STATE{(\textbf{ODE-step}) Compute $v^*$, $s^m$ by solving
            \begin{subequations}
                \label{eq:ode_timedisc}
                \begin{alignat}{1}
                    v^{*}(x) - v^{m-1}(x)   & = -\Delta t I_{\text{ion}}(v^{m-1}(x),s^{m-1}(x)), \\
                    s^{m+1}(x) - s^{m-1}(x) & = \Delta t F(v^{m-1}(x),s^{m-1}(x)).
                \end{alignat}
            \end{subequations}}
        \STATE{(\textbf{PDE-step}) Compute \(u^m_i\), \(u^m_e\),\(v^m\) by solving
            \begin{subequations}
                \label{spatial_emi}
                \begin{alignat}{2}
                    -\nabla \cdot \left( \sigma_e \nabla u^m_e \right) & = 0                                                          &  & \quad \quad \text{in }\Omega_e \label{eq:spatial_Emi1},   \\
                    -\nabla \cdot \left( \sigma_i \nabla u^m_i \right) & = 0                                                          &  & \quad \quad \text{in }\Omega_i , \label{eq:spatial_Emi2}  \\
                    \sigma_e \nabla u^m_e \cdot \boldsymbol{n_e}       & = -\sigma_i \nabla u^m_i \cdot \boldsymbol{n_i} \equiv I^m_m &  & \quad \quad \text{on } \Gamma,\label{eq:spatial_Emi3}     \\
                    u^m_i - u^m_e                                      & = C_{m}^{-1}\Delta t I^m_m +g                                &  & \quad \quad \text{on } \Gamma,\label{eq:spatial_Emi4}     \\
                    u^m_e                                              & = 0                                                          &  & \quad \quad \text{on }\partial\Omega,\label{spatial_Emi5}
                \end{alignat}
            \end{subequations}}
        where \(g=v^*\).
        \ENDFOR
        \RETURN \(u_e\), \(u_i\) and \(v\).
    \end{algorithmic}
\end{algorithm}
\section{Unfitted discretizations of the EMI PDE step}
\label{section:spatial}

The main focus of the remainder of this paper is discretizing the PDE system \eqref{spatial_emi} of the PDE-step in \cref{alg:time_disc}. This system takes the form of two Poisson problems coupled by a Robin-type interface condition, for which several finite element methods have been proposed. Belgacem et al.~\cite{belgacem2015compositemedia} modelled contact resistance problems via systems with the same structure, presenting two different formulations. The first uses only the potentials as unknowns, while the second is a hybrid dual formulation introducing also the current densities. Using the nomenclature introduced by Kuchta et al.~\cite{EMIchapter5}, we will refer to the first one as the single-dimensional primal formulation. Two other alternatives were presented in by Tveito et al.~\cite{tveito2017cell}, a multi-dimensional primal formulation with the potentials and the current across the membrane as unknowns, and a multi-dimensional mixed formulation where in addition also the current densities were unknowns. An unfitted method based on the single-dimensional primal formulation was presented by Ji et al.~\cite{jiUnfittedFiniteElement2017}. We first review this formulation, before presenting a cut finite element formulation based on the multi-dimensional primal formulation.

\subsection{A single-dimensional primal formulation of the EMI PDEs}
Following the presentation in \cite{EMIchapter5}, a weak formulation of \eqref{spatial_emi} can be derived as follows. Define the function spaces
\begin{equation*}
    V_i = H^1(\Omega_i), \quad  V_e =H^1_{0,\partial \Omega}(\Omega_e),
\end{equation*}
where the vanishing boundary-trace space is given by
\begin{align}
    H_{0,\partial \Omega}^1(\Omega_e) = \{ v_e \in  H^1(\Omega_e) \, | \,  \mathrm{tr} \, v_e = 0 \text{ on } \partial \Omega \},
\end{align}
and let \(V = V_i \times V_e\). We multiply~\eqref{eq:spatial_Emi2} by a test function \(v_i \in V_i\) and integrate over \(\Omega_i\), and multiply \eqref{eq:spatial_Emi1} with a test function \(v_e \in V_e\) and integrate over \(\Omega_e\). Integration by parts then gives
\begin{subequations}
    \label{sp_halfweak}
    \begin{alignat}{2}
        \int_{\Omega_e} \sigma_e \nabla u_e \cdot \nabla v_e \dx -\int_{\Gamma} \sigma_e \nabla u_e \cdot \boldsymbol{n_e} v_e \ds & =0 , \\
        \int_{\Omega_i} \sigma_i \nabla u_i \cdot \nabla v_i \dx -\int_{\Gamma} \sigma_i \nabla u_i \cdot \boldsymbol{n_i} v_i \ds & =0.
    \end{alignat}
\end{subequations}
Inserting \eqref{eq:spatial_Emi3} and \eqref{eq:spatial_Emi4} into \eqref{sp_halfweak} gives the abstract problem: find \(u \in V\)  such that
\begin{equation}
    \label{eq:single_weak}
    a(u,v) = l(v) \quad \forall v \in V,
\end{equation}
where the bilinear form \(a(\cdot,\cdot)\) and the linear form \(l(\cdot)\) are given by
\begin{equation}
    a(u,v)= \int_{\Omega_e} \sigma_i \nabla u_i\cdot \nabla v_i \dx  +\int_{\Omega_i} \sigma_e \nabla u_e \cdot \nabla v_e \dx
    +\frac{C_m}{\Delta t}\int_{\Gamma} (u_e - u_i) (v_e -v_i ) \ds,
    \label{eq:emi_fem_bilinear}
\end{equation}
\begin{equation}
    l(v) = \frac{C_m}{\Delta t}\int_{\Gamma} g (v_i- v_e) \,\ds. \label{variational_rhs}
\end{equation}

The bilinear form \eqref{eq:emi_fem_bilinear} induces the following natural energy norm for \(v=(v_i,v_e) \in V = V_i \times V_e\),
\begin{equation}
    \label{eq:emi_norm}
    \norm{v}^2_{a} \vcentcolon = \sigma_e \norm{ \nabla v_e}_{L^2(\Omega_e)}^2 + \sigma_i\norm{\nabla v_i}_{L^2(\Omega_i)}^2 + \frac{C_m}{\Delta t} \norm{v_e - v_i}^2_{\Gamma} .
\end{equation}
The weak formulation is clearly coercive and bounded with respect to \(\norm{\cdot}_a\), and therefore the formulation \eqref{eq:single_weak} is well-posed
thanks to the Lax-Milgram theorem.

\subsubsection{Cut finite element formulation}
We discretize the formulation~\eqref{eq:single_weak} using a cut finite element method (CutFEM)~\cite{burmanCutFEMDiscretizingGeometry2015} approach to allow for a more flexible handling of the membrane interface
geometry. We start from a structured background mesh \(\widetilde{\T}_h\) consisting of either $n$-simplices or $n$-cubes which
cover \(\bar{\Omega}\)  (see \cref{fig:mesh_single}).
Then the active background mesh associated with each of the two subdomains is given by the collections of elements which intersect the corresponding physical domain,
\begin{equation}
    \T_{h,j} = \{T \in \widetilde{\T}_h \, \vert \: T  \cap \overset{\circ}{\Omega}_j \neq \emptyset \}, \quad j \in \{i, e\},
\end{equation}
where  \( \overset{\circ}{\Omega}_i=\Omega_i \setminus \Gamma\), and  \(\overset{\circ}{\Omega}_e=\Omega_e \setminus (\Gamma\ \cup \partial\Omega) \).
\begin{figure}
    \begin{subfigure}[b]{0.3\textwidth}
        \centering
        \begin{overpic}[width=\textwidth]{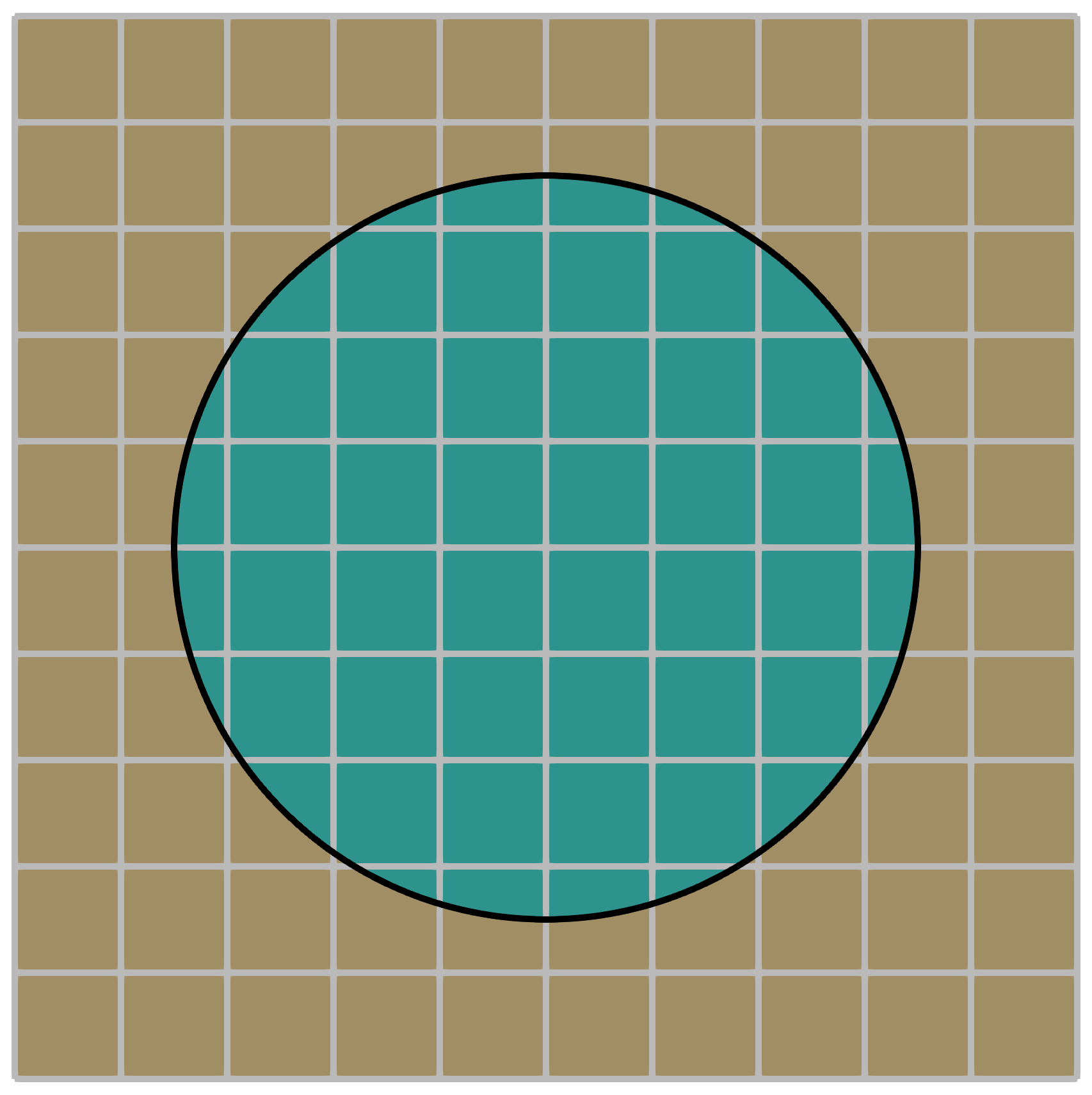}
        \end{overpic}
    \end{subfigure}
    \hfill
    \begin{subfigure}[b]{0.3\textwidth}
        \centering
        \begin{overpic}[width=\textwidth]{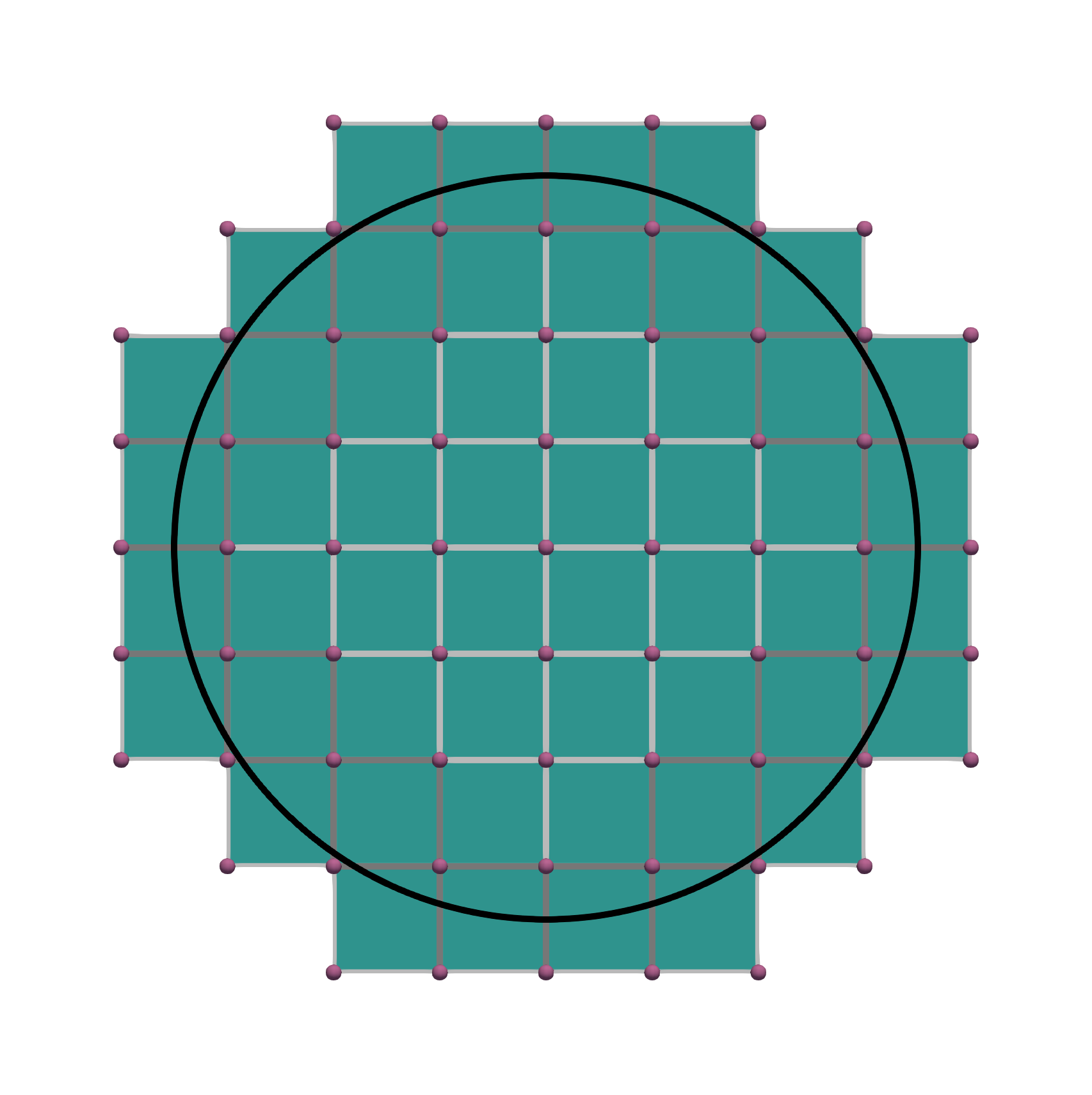}
            \put(41,47){\huge\(\T_{h,i}\)}
        \end{overpic}
    \end{subfigure}
    \hfill
    \begin{subfigure}[b]{0.3\textwidth}
        \centering
        \begin{overpic}[width=\textwidth]{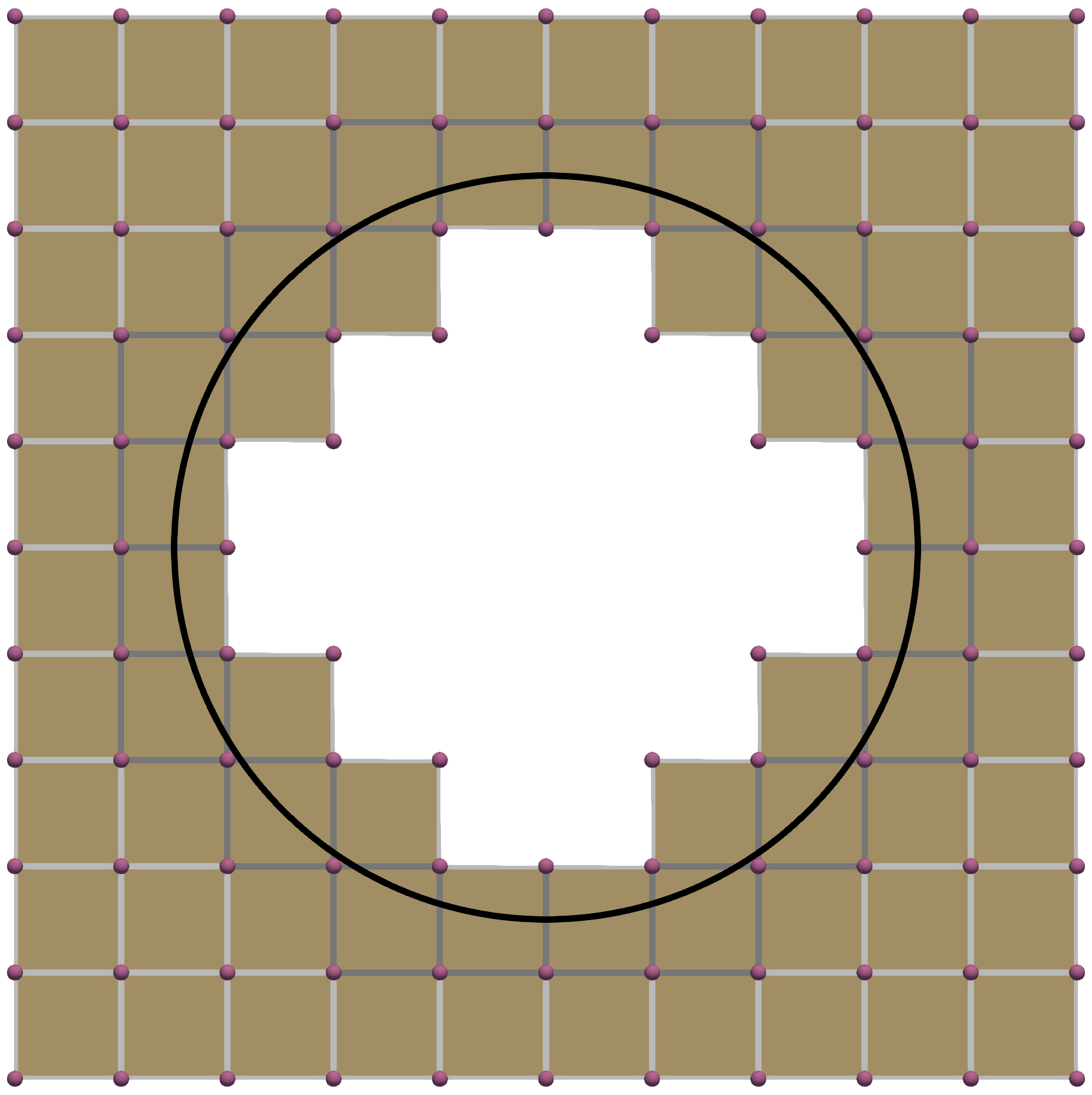}
            \put(70,80){\huge\(\T_{h,e}\)}
        \end{overpic}
    \end{subfigure}
    \hfill
    \caption{Illustration of the total computational domain(left) the single-dimensional discretization and the (left) intracellular and (right) extracellular computational domains, with corresponding ghost facets. }
    \label{fig:mesh_single}
\end{figure}
The sets of interior faces are denoted by
\begin{equation}
    \mathcal{F}_{h,i} = \{F = T^+ \cap T^-  \, \vert \: T^+,T^- \in \mathcal{T}_{h,i} \, \mathrm{ and } \,  T^+ \neq T^{-} \}.
\end{equation}
Also, the set of ghost penalty faces \(\F_h^g\), is given by the set of interior faces in the active mesh belonging to elements that are intersected by the boundary $\Gamma$,
\begin{equation}
    \F_{h,i}^g = \{ F \in \F_{h,i} : T^+ \cap \Gamma \neq \emptyset \lor T^- \cap \Gamma \neq \emptyset \} .
\end{equation}
Let the jump across an interior face $F \in \F_h$ be
\begin{equation}
    [w]\vert_F = w^+_F - w^-_F,
\end{equation}
where \(w^\pm(x)= \lim_{t \to 0^+} w(x \mp t \mathbf{n}_F) \) for some chosen unit face normal $\mathbf{n}_F$.

On a given mesh $\T_h$, the corresponding space of piecewise continuous polynomials of degree \(k\) is denoted by
\begin{equation} \mathbb{P}^k_\mathrm{c}(\mathcal{T}_h) =  \{ v \in C(\Omega_h) : v\vert_T \in \mathbb{P}^k(T)\; \forall T \in \T_h \},
\end{equation}
where \(\Omega_h = \bigcup_{T\in \T_h} T,\)
and \(\mathbb{P}^k(T)\) is the set for all (multivariate if applicable) polynomials with degree smaller than or equal to \(k\) on \(T\).
As discretization of the function space $V_i$ and $V_e$, we define the finite element spaces
\begin{equation}
    V_{h,i} = \mathbb{P}^k_\mathrm{c}(\T_{h,i}), \quad V_{h,e}=\{v \in \mathbb{P}^k_\mathrm{c}(\T_{h,e}) \, | \, v\vert_{\partial \Omega}=0\},
\end{equation}
and the resulting total finite element space \(V_{h} = V_{h,i} \times V_{h,e}\).

Since $V_h \subset V$, the Lax-Milgram theorem again ensures that the resulting discrete problem where we seek \(u_h=(u_{h,i},u_{h,e}) \in V_h\) such that
\begin{equation}
    \label{eq:weak_cutfem_unstabilized}
    a(u_h,v)  = l(v) \quad \forall\; v \in V_h
\end{equation}
is well-posed. Furthermore, optimal convergence estimates can be derived~\cite{jiUnfittedFiniteElement2017}, even without any CutFEM related addition stabilization~\cite{liHighOrderInterfacepenalty2022}, despite using an unfitted finite element approach. The main reason is that the interface conditions are included naturally  and thus leaves the bilinear form unchanged in the discrete formulation,
so that the coercivity properties are automatically inherited from continuous weak formulation~\eqref{eq:single_weak}.
This is in contrast to problems where the weak form needs to be altered because of the imposition of essential boundary or interface
conditions, see e.g.~the discussion in~\cite{burmanCutFEMDiscretizingGeometry2015,gurkanStabilizedCutDiscontinuous2019}.

However, without any stabilization or further intervention, the
possible presence of so-called small cut elements $T\cap \Omega_j $ where $|T \cap \Omega_j|
    \ll |T|$ can lead to a nearly singular system matrix and consequently
excessively high condition numbers,
which in turn can dramatically limit the
use and development of efficient iterative solvers.
To render the discretization scheme truly geometrical robust, we want
to ensure that the resulting condition numbers only depends as usual on
the size of the linear system but not on the relative positioning
of the membrane interface within the background mesh.
We therefore add to~\eqref{eq:weak_cutfem_unstabilized} the following ghost penalty
\cite{burmanFictitiousDomainFinite2012,MassingLarsonLoggEtAl2013a,gurkanStabilizedCutDiscontinuous2019},
\begin{equation}
    \label{eq:ghost_single}
    g_{h}(v,w)=  g_{h,i}(v_i,w_i) + g_{h,e}(v_e,w_e),
\end{equation}
where
\begin{equation}
    g_{h,j}(v_j,w_j)=\sum_{m=0}^k  \gamma_m h^{2m+1} \left([\partial_n^m v_j],[\partial_n^m w_j]\right)_{ \F_{h,j}^g }.
\end{equation}
The $h$-scaling of the ghost penalty employed here is chosen such that
the physical $L^2$-norm augmented with the semi-norm induced by the
ghost penalty is equivalent to the $L^2$-norm on the corresponding
active mesh,
$$
    \|v_{h,j}\|_{\T_{h,j}}^2 \sim
    \|v_{h,j}\|_{\Omega_j}^2 + |v_{h,j}|^2_{g_{h, j}} \quad \forall v_{h,j} \in V_{h,j}, \quad j \in \{i, e\},
$$
see, e.g.~\cite{gurkanStabilizedCutDiscontinuous2019} for a detailed discussion.

The final stabilized CutFEM single-dimensional formulation is then: find \(u_h=(u_{h,i},u_{h,e}) \in V_h\) such that
\begin{equation}
    \label{eq:weak_cutfem}
    A_h(u_h,v) \equiv a(u_h,v) +  g_{h}(u_h,v)  = l(v),
\end{equation}
for all \(v=(v_i,v_e) \in V_h\). This formulation converges at optimal rates in the \(H^1\)- and \(L^2\)-norms, and the corresponding condition number is bounded by \(\mathcal{O}(h^{-2})\)~\cite{jiUnfittedFiniteElement2017}.

\subsection{A multi-dimensional weak formulation}
Next, we discuss how the EMI PDEs \eqref{spatial_emi} can be cast into a penalized saddle point problem using the multi-dimensional primal formulation presented in~\cite{EMIchapter5}.
As for the single-dimensional formulation, we start from the function spaces
\begin{align}
    V_i := H^1(\Omega_i) , \quad V_e = H_{0}^1(\Omega_e), \quad V= V_i \times V_e.
\end{align}
Next, we multiply \eqref{eq:spatial_Emi1} and \eqref{eq:spatial_Emi2} with test functions from \(V_e\) and \(V_i\) and use integration by parts to obtain
\begin{align}
    \int_{\Omega_e} \sigma_e \nabla u_e\cdot \nabla v_e \dx - \int_{\Gamma} v_e I_m \ds & =0, \\
    \int_{\Omega_i} \sigma_i \nabla u_i\cdot \nabla v_i \dx + \int_{\Gamma} v_i I_m \ds & =0.
\end{align}
In contrast to the single-dimensional formulation, we now leave \(I_m\) as a separate unknown. Define the additional function spaces
\begin{equation}
    Q =H^{-\frac{1}{2}}(\Gamma),  \quad Q_c =  L^2(\Gamma),
\end{equation}
and let \(I_m \in Q_c\).
Multiply \eqref{eq:spatial_Emi4} by \(j_m \in Q_c\) to obtain the last equation,
\begin{align}
    \int_{\Gamma} (u_i-u_e) j_m \,\ds  - \frac {\Delta t}{C_m}\int_{\Gamma} I_m j_m \,\ds & = \int_{\Gamma} g j_m \,\ds.
    \label{eq:spatial_Emi4-weak}
\end{align}
Note that the main reason for choosing the subspace $Q_c \subset Q$ as function space for $I_m, j_m$
is to make sense of the second term appearing in the left-hand side of~\eqref{eq:spatial_Emi4-weak}.

The resulting multi-dimensional primal formulation is then: find \(u_i \in V_i \), \(u_e \in V_e\), \(I_m \in Q_c\) such that%
\begin{subequations}
    \label{eq:multi_weak}
    \begin{align}
        \int_{\Omega_e} \sigma_e \nabla u_e\cdot \nabla v_e \,\dx -\int_{\Gamma} v_e I_m \,\ds & =0 , \label{eq:multi_weak1}                        \\
        \int_{\Omega_i} \sigma_i \nabla u_i\cdot \nabla v_i \,\dx +\int_{\Gamma} v_i I_m \,\ds & =0,                                                \\
        \int_{\Gamma} (u_i-u_e) j_m \,\ds  - \frac{\Delta t}{C_m}\int_{\Gamma} I_m j_m \,\ds   & = \int_{\Gamma} g j_m  \ds \label{eq:multi_weak3},
    \end{align}
\end{subequations}
for all \(v_i \in V_i\),  \(v_e \in V_e\) and \(j_m \in Q_c\). After defining the corresponding bilinear and linear forms
\begin{align*}
    a(u, v) & = \sigma_e (\nabla u_e , \nabla v_e )_{\Omega_e} + \sigma_i (\nabla u_i , \nabla v_i )_{\Omega_i}, \\
    b(v, I) & = (v_i - v_e, I)_{\Gamma}, \qquad
    c(I, j) = (I, j)_{\Gamma}, \qquad
    l(j) = (g, j)_{\Gamma},
\end{align*}
and the associated total bilinear form 
\begin{align}
    A(u, I; v, j) = a(u,v) + b(u, j) + b(v, I) - \frac{\Delta t}{C_m} c(I, j),
\end{align}
we can write the problem in the form of a penalized saddle point problem \cite{braessStabilitySaddlePoint1996}: find \((u,I_m) \in V \times Q_c\) such that
\begin{align}
    \label{eq:final_weak_multi}
    A(u,I_m;v,j_m) = l(j_m)
\end{align}
for all \((v,j_m) \in V \times Q_c\).

Next, we show that problem~\eqref{eq:final_weak_multi} is well-posed. To analyze its well-posedness,
we introduce the following norms on the spaces $V$ and $Q_c,$
\begin{align}
    \norm{u}^2_{V}      & = \sigma_e \norm{\nabla u_e}^2_{\Omega_e} +  \sigma_i \norm{\nabla u_i}^2_{\Omega_i} +\norm{u_i-u_e}^2_{\Gamma} , \\
    \norm{j_m}^2_{Q}  & = \norm{j_m}^2_{-1/2,\Gamma}, \\
    \norm{j_m}^2_{Q_c}  & = \norm{j_m}^2_{\Gamma},                                                                                          \\
    \vertiii{(v,j_m)}^2 & = \norm{v}^2_{V} + \norm{j_m}^2_Q +  \frac{{\Delta t}}{C_m} \norm{j_m}^2_{Q_c}.
\end{align}
The following lemma then shows that the weak formulation~\eqref{eq:final_weak_multi} indeed satisfies an inf-sup condition.
\begin{lemma}
    For \(0\leqslant \frac{{\Delta t}}{C_m} \leqslant 1\) and \(g \in L^2(\Gamma)\) the saddle point formulation \cref{eq:final_weak_multi} is inf-sup stable, i.e., there exists a \(\gamma >0 \) such that
    \begin{align*}
        \inf_{(u,I_m)\in V \times Q_c} \sup_{(v,j_m)\in V \times Q_c} \frac{A(u,I_m; v, j_m)}{\vertiii{(u,I_m)}\cdot\vertiii{(v,j_m)}} \geqslant \gamma.
\end{align*}
\end{lemma}
\begin{proof}
Well-posedness for penalized saddle point problems is discussed in \cite{braessStabilitySaddlePoint1996}. 
More specifically, if the problem without the penalization term satisfies the LBB conditions, meaning that if the unpenalized problem is well-posed, the penalized problem is stable if in addition the following requirement is satisfied \cite[Lemma 3]{braessStabilitySaddlePoint1996},
\begin{align}
    \frac{a(u,u)}{\norm{u}_V} + \sup_{\I_m \in Q_c} \frac{b(u,I_m)}{\norm{I_m}_Q + \varepsilon \norm{I_m}_{Q_c}} \geqslant \alpha \norm{u}_V
    \quad \forall u \in V
    , \label{eq:requirement_braess}
\end{align}
for some \(\alpha>0\), where \(\varepsilon =  \sqrt{\frac{{\Delta
t}}{C_m}}\) need to satisfy \(0\leqslant\varepsilon\leqslant1\). The
unpenalized version of~\eqref{eq:final_weak_multi} is shown to be
well-posed in \cite{lamichhaneMortarFiniteElements2004}. It
remains to show that \eqref{eq:requirement_braess} holds.
We start by noting that we have a Gelfand triple with continuous
embeddings
\(H^{\frac{1}{2}}(\Gamma) \subset L^2(\Gamma) \subset H^{-\frac{1}{2}}(\Gamma)\)
and that the jump \(  u_i\vert_\Gamma -
u_e\vert_\Gamma\) of a function \(u \in V\) is in \( H^{\frac{1}{2}}(\Gamma)\). 
This implies that  \(\norm{u_i - u_e}_Q \leqslant C_1 \norm{u_i -
u_e}_{Q_c} \leqslant C_2 \norm{u_i - u_e}_{1/2,\Gamma}\) for some positive constants 
\(C_1, C_2\), and thus we conclude that
\begin{align}
    &\frac{a(u,u)}{\norm{u}_V} + \sup_{\I_m \in Q_c} \frac{b(u,I_m)}{\norm{I_m}_Q + \varepsilon \norm{I_m}_{Q_c}} \\ 
    &\geqslant  \frac{ \sigma_e \norm{\nabla u_e}^2_{\Omega_e} +  \sigma_i \norm{\nabla u_i}^2_{\Omega_i}}{\norm{u}_V} +  \frac{b(u,u_i\vert_\Gamma - u_e\vert_\Gamma)}{\norm{u_i- u_e}_Q + \varepsilon \norm{u_i- u_e}_{Q_c}}\\
    &\geqslant \frac{ \sigma_e \norm{\nabla u_e}^2_{\Omega_e} +  \sigma_i \norm{\nabla u_i}^2_{\Omega_i}}{\norm{u}_V} + \frac{\norm{u_i- u_e}^2_\Gamma}{C\norm{u_i- u_e}_\Gamma} 
    \geqslant \alpha \norm{u}_V.
\end{align}
\end{proof}

\subsubsection{Cut finite element formulation}
Next, we devise a CutFEM formulation for the multi-dimensional primal formulation.
As for the single-dimensional problem, the two active background meshes \(\T_{h,i}\) and \(\T_{h,e}\)
are used to define approximation spaces for the intra and extra-cellular
potentials (see \Cref{fig:mesh_multi_theory}).
To discretize the surface-bounded current \(I_m\),
we now also define an active mesh associated with the membrane surface,
\begin{equation}
    \T_{h,\Gamma} = \{T \in \widetilde{\T}_h \, | \: T \cap \Gamma \neq \emptyset \},
\end{equation}
and the corresponding set of interior faces,
\begin{equation*}
    \mathcal{F}_{h,\Gamma} = \{F = T^+ \cap T^- \, \vert \: \,  T^+,T^-\in \T_{h,\Gamma} \, \mathrm{and} \,  T^+ \neq T^{-}\}.
\end{equation*}
\begin{figure}
    \begin{subfigure}[b]{0.3\textwidth}
        \centering
        \begin{overpic}[width=\textwidth]{./figures/intramesh.png}
            \put(41,47){\huge\(\T_{h,i}\)}
        \end{overpic}
    \end{subfigure}
    \hfill
    \begin{subfigure}[b]{0.3\textwidth}
        \centering
        \begin{overpic}[width=\textwidth]{./figures/extramesh.png}
            \put(70,80){\huge\(\T_{h,e}\)}
        \end{overpic}
    \end{subfigure}
    \hfill
    \begin{subfigure}[b]{0.3\textwidth}
        \centering
        \begin{overpic}[width=\textwidth]{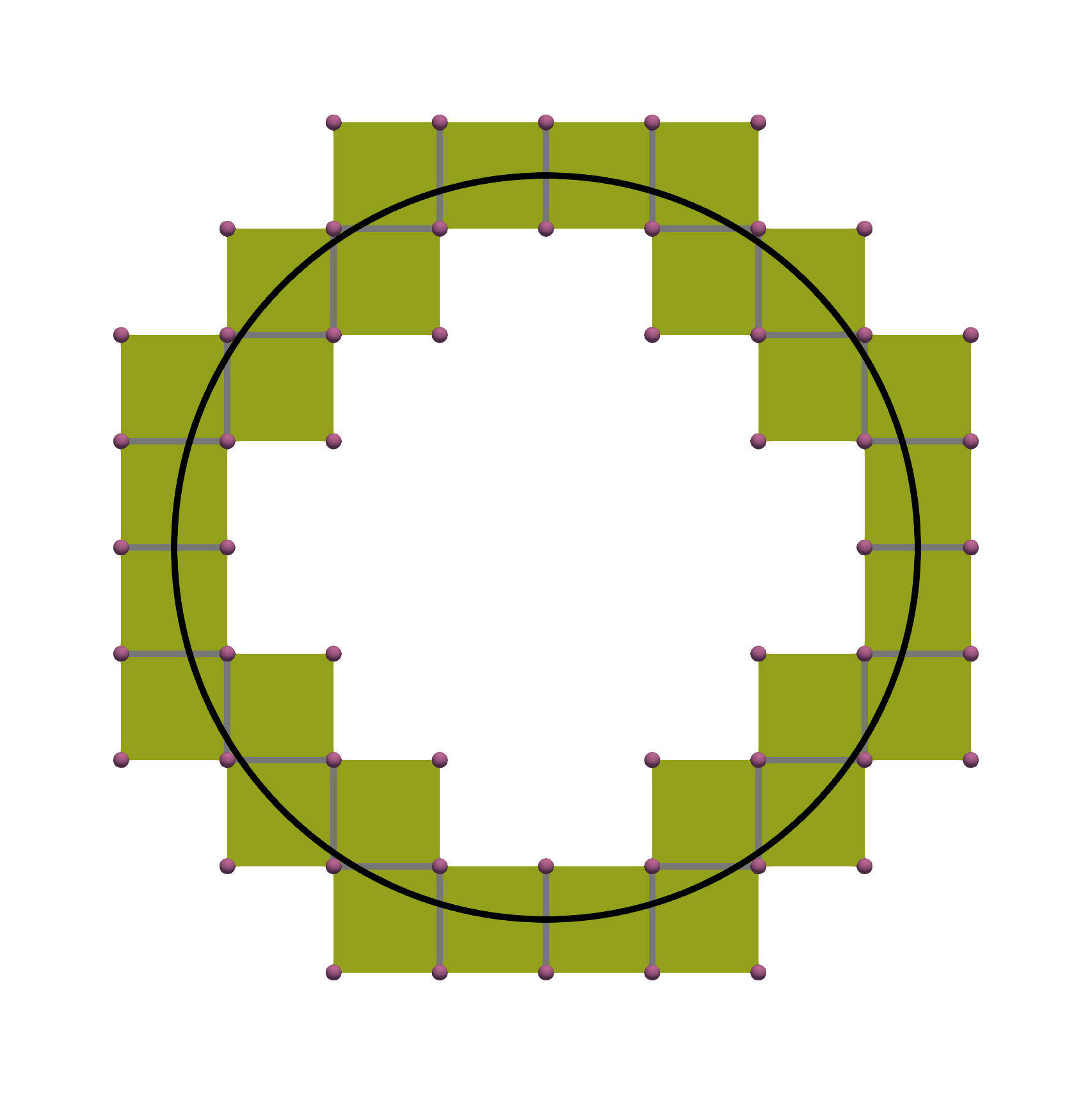}
            \put(42,45){\huge\(\T_{h,\Gamma}\)}
            \put(58,55){\huge\vector(1,1){10}}
        \end{overpic}
    \end{subfigure}
    \caption{Illustration of the computational domains for the multi-dimensional discretization with the (left) intracellular, (middle) extracellular and (right) interface computational domains.}%
    \label{fig:mesh_multi_theory}%
\end{figure}
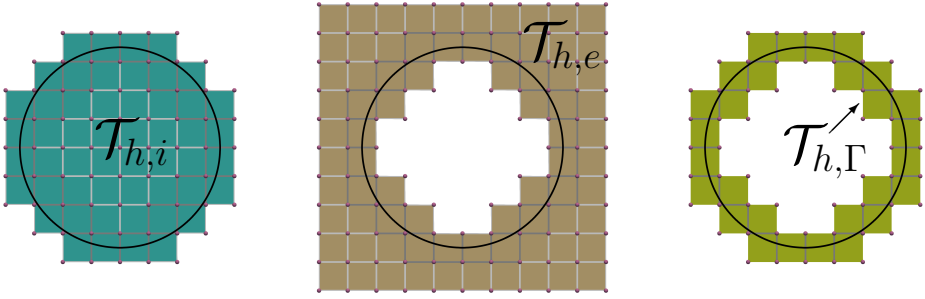%
As finite element spaces, we employ first-order continuous piecewise linear elements for the potentials and constant linear elements for the current,
\begin{equation*}
    V_{h,i} = \mathbb{P}^1_\mathrm{c}(\T_{h,i}), \quad V_{h,e}=\{v \in \mathbb{P}^1_\mathrm{c}(\T_{h,e}) : v\vert_{\partial \Omega}=0\}, \quad Q_h =  \mathbb{P}^0_\mathrm{dc}(\T_{h,\Gamma}),
\end{equation*}
and as before, we set \(V_h =  V_{h,i} \times  V_{h,e}\).
We emphasize that the finite element space for the current is defined
on the active background mesh \(\T_{h,\Gamma}\) and not on \(\Gamma\) only.

A preliminary and unstabilized
unfitted multi-dimensional finite element formulation of the EMI PDEs
would be: find \((u_h,I_{m,h}) \in V_h \times Q_h\) such that
\begin{subequations}
    \label{eq:finite_formulation_saddle}
    \begin{alignat}{2}
        a(u_h,v_h) + b(v_h, I_{m,h})                             & = 0,                &  & \forall v_h\in V_h,      \\
        b(u_h,j_{m,h}) -\frac{\Delta t}{C_m} c(I_{m,h} ,j_{m,h}) & = l(j_{m,h}), \quad &  & \forall j_{m,h} \in Q_h,
    \end{alignat}
\end{subequations}
with total bilinear form
\begin{align*}
    A(u_h,I_{h};v_h,j_{h}) & = a(u_h,v_h) + b(v_h,I_{h}) + b(u_h, j_{h}) - \frac{\Delta t}{C_m} c(I_{h},j_{h}).
\end{align*}

In order to render this discretization scheme both stable and geometrically robust,
we will add two stabilization terms. First, as for the single-dimensional formulation, we add the ghost penalty \(g_h\) defined
in~\eqref{eq:ghost_single} to ensure that the presence of small cut elements does not
lead to solely negligible contributions to total bilinear form \(a\). Otherwise, the condition number of the system matrix would be highly dependent on the particular cut
configuration.
Second, in the spirit of~\cite{BurmanHansbo2009,Burman2013}, we
define a stabilization operator
\begin{align*}
    s_h(I_{h},j_{h}) = \phi \sum_{F \in \F_{h,\Gamma}} ([I_{h}],[j_{h}])_F,
\end{align*}
where
\begin{align*}
    \phi= \max\bigg\{\frac{\Delta t}{C_m},h\bigg\},
\end{align*}
which ensures that the unperturbed saddle point problem with $\Delta t = 0$
satisfies an inf-sup condition.

This results in the stabilized bilinear form,
\begin{align}
    A_h(u_h,I_{h};v_h,j_{h}) & = A(u_h,I_{h};v_h,j_{h}) + g_h(u_h,v_h) -s_h(I_{h},j_{h}),
\end{align}
and final stabilized CutFEM discretization for the multi-dimensional primal formulation:
find \((u_h,I_{m,h}) \in V_h \times Q_h\) such that
\begin{align}
    \label{eq:multi_cutfem}
    A_h(u_h,I_{m,h};v_h,j_{m,h}) & = l(j_{m,h})
\end{align}
for all \((v_h,j_{m,h}) \in V_h \times Q_h\).

\section{Unfitted discretizations of the interface ODE step}
\label{sec:ode}

When solving the EMI model on a fitted mesh, the ODE-system is
naturally solved on each of the interface mesh vertices, degrees of
freedom or quadrature
points~\cite{tveito2017cell}. For an unfitted
mesh, there are no longer any mesh points on the interface. To
overcome this problem, we here introduce a new unfitted discretization
of ODEs based on a stabilized mass matrix approach, similar to a
finite element discretization in time of parabolic PDEs.

We begin by introducing a weak formulation based on a
stabilized \(L^2\)-projection problem on the surface. Let \(f\in
L^2(\Gamma)\) be a given function on the surface \(\Gamma\).  For the
multi-dimensional formulation, we defined the active mesh associated
with the cells that are cut by
\(\Gamma\) as \(\T_{h,\Gamma}\). Now we use $\T_{h,\Gamma}$ to introduce the
discrete function space \(Q_{h,\mathrm{ode}} =
\mathbb{P}^1_\mathrm{c}(\T_{h,\Gamma})\) to be the space of piecewise
(bi-)linear continuous functions on \(\T_{h,\Gamma}\). The simple and
natural way to define the \(L^2\)-projection of \(f\) would be: find
\(u_h \in Q_{h,\mathrm{ode}} \) such that
\begin{equation}
    \label{eq:un_projection}
    m_h(u_h,w) = (f,w)_\Gamma \quad \forall w\in Q_{h,\mathrm{ode}},
\end{equation}
where \(m_h(u_h,w) = (u_h,w)_\Gamma\). However, the discrete formulation \eqref{eq:un_projection} suffers from similar issues arising in other unstabilized unfitted formulations. In particular, the mass matrix associated with the problem has potentially large condition numbers
which in addition are highly dependent on how the background mesh and the membrane surface intersect.

To remedy this problem, several different stabilization operators \(s_h\) can be added \cite{burmanCutFiniteElement2018}. The resulting stabilized \(L^2\)-projection is given by: find \(u_h \in V_h \) such that
\begin{equation}
    \label{eq:projection_stable}
    M_h (u_h,v) \equiv  (u_h,w)_\Gamma + s_h(u,w)
    = (f,w)_{\Gamma} \quad \forall w\in V_h .
\end{equation}
In this work, we consider two realizations of $s_h(\cdot,\cdot)$:
a face-based stabilization,
\begin{equation}
    \label{ghost_penalty_time1}
    s^1_h(v,w)=\ \sum_{F\in \F_{\Gamma,h}} \gamma_{b} h_F^{2}([\partial_n v],[\partial_n w])_F,
\end{equation}
and a volume normal gradient stabilization
\begin{equation}
    \label{ghost_penalty_time2}
    s^2_h(v,w)=\gamma_{b} h_F([\partial_{n_{\Gamma}} v],[\partial_{n_{\Gamma}} v])_{\T_{h,\Gamma}},
\end{equation}
both proposed in \cite{burmanCutFiniteElement2018}, where \(\gamma_b\) is some positive parameter. Note that \eqref{ghost_penalty_time1} resembles the stabilization used for the PDE in the previous section. Importantly, note that \(s^1_h\) is only suitable for first-order elements. The second stabilization can be used for higher-order elements but is harder to implement if the normal-field of \(\Gamma\) is not given naturally on a volume mesh.

Now, we can use the stabilized \(L^2\)-projection to solve the ODE system on an unfitted surface.
A weak formulation is obtained by multiplying the time discretized ODEs \eqref{eq:ode_timedisc}
by test functions \(w_1,w_2 \in \T_{h,\Gamma}\), yielding
\begin{alignat}{2}
    (v^{n+1} - v^{n},w_1)_{\Gamma} & = -\Delta t\big(I_{\text{ion}}(v^{n},s^{n}),w_1\big)_{\Gamma}, \\
    (s^{n+1} - s^{n},w_2)_{\Gamma} & = \Delta t\big(F(v^n,s^n),w_2\big)_{\Gamma}.
\end{alignat}
Adding the same stabilization as for the \(L^2\)-projection problem, we arrive at the following system of equations to be solved for each time step \(n=1, \dots,N\),
\begin{alignat}{2}
    M_h(v^{n+1} ,w_1)  & =  m_h\big(v^{n} - \Delta t I_{\text{ion}}(v^{n},s^{n}),w_1\big), \\
    M_h (s^{n+1} ,w_2) & = m_h\big(s^{n} + \Delta t F(v^n,s^n),w_2\big).
\end{alignat}
The initial conditions \(v^{0}\) and \(s^{0}\) also needs to be specified. The natural approach is to use the stabilized \(L^2\)-projection of the initial value such that
\begin{alignat}{2}
    M_h(v^{0},w_1) & = \Delta t m_h(v(t_0)), \\
    M_h(s^{0},w_2) & = \Delta t m_h(s(t_0)).
\end{alignat}
In our numerical experiments we will have the initial conditions given as explicit expressions, and we will therefore rather use interpolation to define the initial value.

\section{Numerical results}
\label{sec:numerical}

We here present numerical experiments using the methods introduced to discretize the EMI model. First, we evaluate the convergence order and geometrical robustness for the PDE- and ODE-steps separately. Next, we perform a convergence study for the full EMI model with an analytical solution. Finally, we run simulations with the coupled EMI Hodgkin-Huxley model.

\subsection{Implementation}
All numerical experiments have been implemented in the Julia-based~\cite{bezansonJuliaFreshApproach2017} finite element framework Gridap~\cite{badiaGridapExtensibleFinite2020, Verdugo2022}. Gridap includes a high-level domain-specific interface in which the weak formulations can be written in near mathematical notation. The plugin GridapEmbedded~\cite{GridapEmbeddedgithub} allows for the use of unfitted meshes, with the boundary of the physical domain defined as a level set function. In addition, the package STLCutters~\cite{STLcuttersgithub} has been used to instead represent the boundary as an STL surface. In all experiments, we set the penalty parameter for both the PDE-step and ODE-step \(\gamma = \gamma_b = 0.1\).

\subsection{Multi-dimensional CutFEM discretization converges at optimal order}
\label{section:experiments_spatial}

We consider a convergence study for the multi-dimensional CutFEM discretization of problem \eqref{spatial_emi}. For numerical studies concerning the single-dimensional formulation, we refer to~\cite{jiUnfittedFiniteElement2017,liHighOrderInterfacepenalty2022}. We use the method of manufactured solutions, and include additional source terms $f_i, f_e$ on the right-hand sides of the first two equations, respectively, in~\eqref{spatial_emi} and consider the potentially non-homogeneous boundary condition $u = u_{bc}$ on $\partial \Omega$.
Generally, we define the intracellular domain based on a level set function \(\varphi\), and the extracellular to be the remaining part of the total domain.
\begin{align}
    \Omega_i = \{(x,y) \in \R^2 \mid \varphi(x,y)<0 \}, \quad \Omega_e=\Omega\setminus\Omega_i.
\end{align}

To begin, we define a total domain by
\(\Omega_1=[-1.8,1.8]\times[-2.05,1.55]\), and let the level set
function be given by
\begin{align}
    \varphi_1(x,y)= x^2+y^2+y\sin((x+1)^2) - 1.5.
    \label{eq:level_set1}
\end{align}
We construct a manufactured solution by setting
\begin{align}
    u_i(x,y) & = \frac{1}{\sigma_i} \sin(0.5\pi x) \cos(0.5\pi y),          \\
    u_e(x,y) & = \frac{1}{\sigma_e} \sin(0.5\pi x) \cos(0.5\pi y),          \\
    I_m(x,y) & = \nabla( \sin(0.5\pi x) \cos(0.5\pi y)) \cdot \mathbf{n_e},
\end{align}
where \(\mathbf{n_e}\) is the inward normal of \(\Omega_i\), and set \(\Delta t=0.2, C_m=1 , \sigma_i =1.5 , \sigma_e=1 \).
The problem is solved on a uniform mesh consisting of \(N^2\) uniform square elements, for which we gradually decrease the element size \(h\) by setting \(N=2^{2+n}\), with \(n=2,\dots,6\). For each mesh refinement level \(n\), the error \(E_n\) between the analytical solution \(u\) and the numerical solution \(u_n\) is calculated in a chosen norm (to be further specified),
\begin{align*}
    E_n = \norm{u-u_n} = \norm{e_u},
\end{align*}
and the experimental order of convergence (EOC) is calculated via
\begin{align*}
    \text{EOC} = \frac{\log(E_{n-1}/E_n)}{\log(h_{n-1}/h_n)}.
\end{align*}

The resulting errors and convergence rates are given in \cref{table:exp1_multi}. We see that the convergence rates are optimal, giving first-order convergence in the \(H^1\)-semi-norm and second-order convergence in the \(L^2\)-norm for \(u\).  The solution for \(N=256\) is depicted in \cref{fig:spatial_solution}.
\begin{table}
    \begin{center}
        \footnotesize
        \caption{Error and convergence rates for the CutFEM multi-dimensional formulation in the spatial convergence study.}
        \label{table:exp1_multi}
        \begin{tabular}{l*{1}{S[table-format=1.2e-2]}*{1}{S[table-format=-1.2]}*{1}{S[table-format=1.2e-2]}*{1}{S[table-format=-1.2]}*{1}{S[table-format=1.2e-2]}*{1}{S[table-format=-1.2]}}
            \hline
            {N} & {\(\norm{e_u}_{\Omega_1\cup\Omega_2}\)} & {EOC} & {\(\abs{e_u}_{1,\Omega_1\cup\Omega_2}\)} & {EOC} & {\(\norm{e_{I_m}}_{\Gamma}\)} & {EOC} \\
            \hline
            16  & 3.42e-02                                &       & 4.84e-01                                 &       & 4.37e-01                      &       \\
            32  & 8.62e-03                                & 1.99  & 2.32e-01                                 & 1.06  & 2.11e-01                      & 1.05  \\
            64  & 2.18e-03                                & 1.98  & 1.14e-01                                 & 1.03  & 1.01e-01                      & 1.06  \\
            128 & 5.34e-04                                & 2.03  & 5.63e-02                                 & 1.02  & 4.51e-02                      & 1.17  \\
            256 & 1.36e-04                                & 1.98  & 2.81e-02                                 & 1.01  & 2.09e-02                      & 1.11  \\
            \bottomrule
        \end{tabular}
    \end{center}
\end{table}
\begin{figure}
    \centering
    \includegraphics[width=0.3\textwidth]{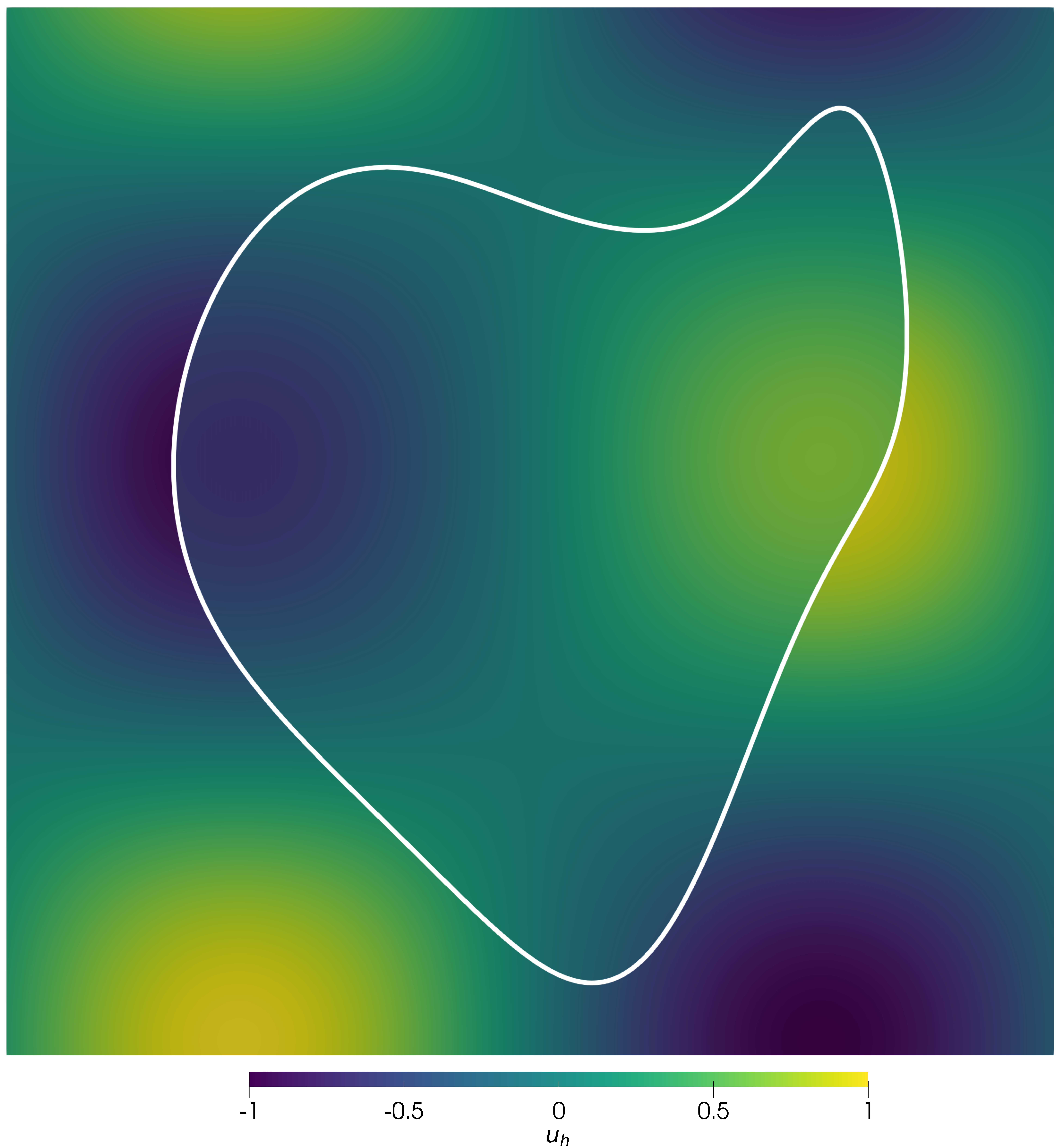}
    \caption{Solution plot for the CutFEM formulations solved on a mesh with \(256\times256\) elements.}
    \label{fig:spatial_solution}
\end{figure}

\subsection{CutFEM condition numbers rely on stabilization for robustness}
Next, we perform a sensitivity study to check how different cut configurations affect the condition number of the stiffness matrix related to the problem. The stiffness matrix \(\mathcal{A}\) based on the discrete form \(A_h\) is defined by the relation
\begin{equation*}
    (\mathcal{A}V,W)_{\R^N} = A_h(v,w) \quad \forall v,w\in V_h.
\end{equation*}
The condition number of the stiffness matrix is defined by
\begin{equation*}
    \kappa(\mathcal{A})=\norm{\mathcal{A}}_{\R^N} \norm{\mathcal{A}^{-1}}_{\R^N},
\end{equation*}
with the corresponding norm defined by
\begin{equation*}
    \norm{\mathcal{A}}_{\R^N}= \sup_{v\in\R^n \setminus \mathbf{0}} \frac{\norm{\mathcal{A}V}_{\R^N}}{\norm{V}_{\R^N}}.
\end{equation*}
Here, we compute the condition numbers with the \texttt{cond method} provided in Julia, with the \(2\)-norm. By repeatedly moving \(\Omega_i\) while assembling the stiffness matrices and computing the associated condition number, we evaluate how sensitive the condition number of the different formulations (non-stabilized versus stabilized) is to the configuration of the cut.

Specifically, let \(\Gamma_\delta\) be a circle with radius \(0.5\), and the center of the circle for step \(m\) be defined by
\begin{align}
    S=\delta\left(\frac{1}{N},\frac{1}{N}\right), \quad \delta =\frac{m}{M_\delta},
\end{align}
where \(M_\delta\) is the number of steps, and \(N\) is the mesh size. We set \(N=32, M_\delta=500, \Delta t = 0.5, \sigma_e = 2, \sigma_1 = 1, C_m=1.\) \cref{fig:sensitivity}(a) shows how the circular interface is moved from the first to the last iteration. The corresponding condition numbers (\cref{fig:sensitivity} (b)) demonstrate that the non-stabilized formulation is very sensitive to the particular cut configurations, in contrast to the stabilized formulation where the condition numbers are practically constant.

\begin{figure}
    \captionsetup[subfigure]{justification=justified, singlelinecheck=false}
    \begin{subfigure}[t]{0.26\textwidth}
        \includegraphics[width=\textwidth]{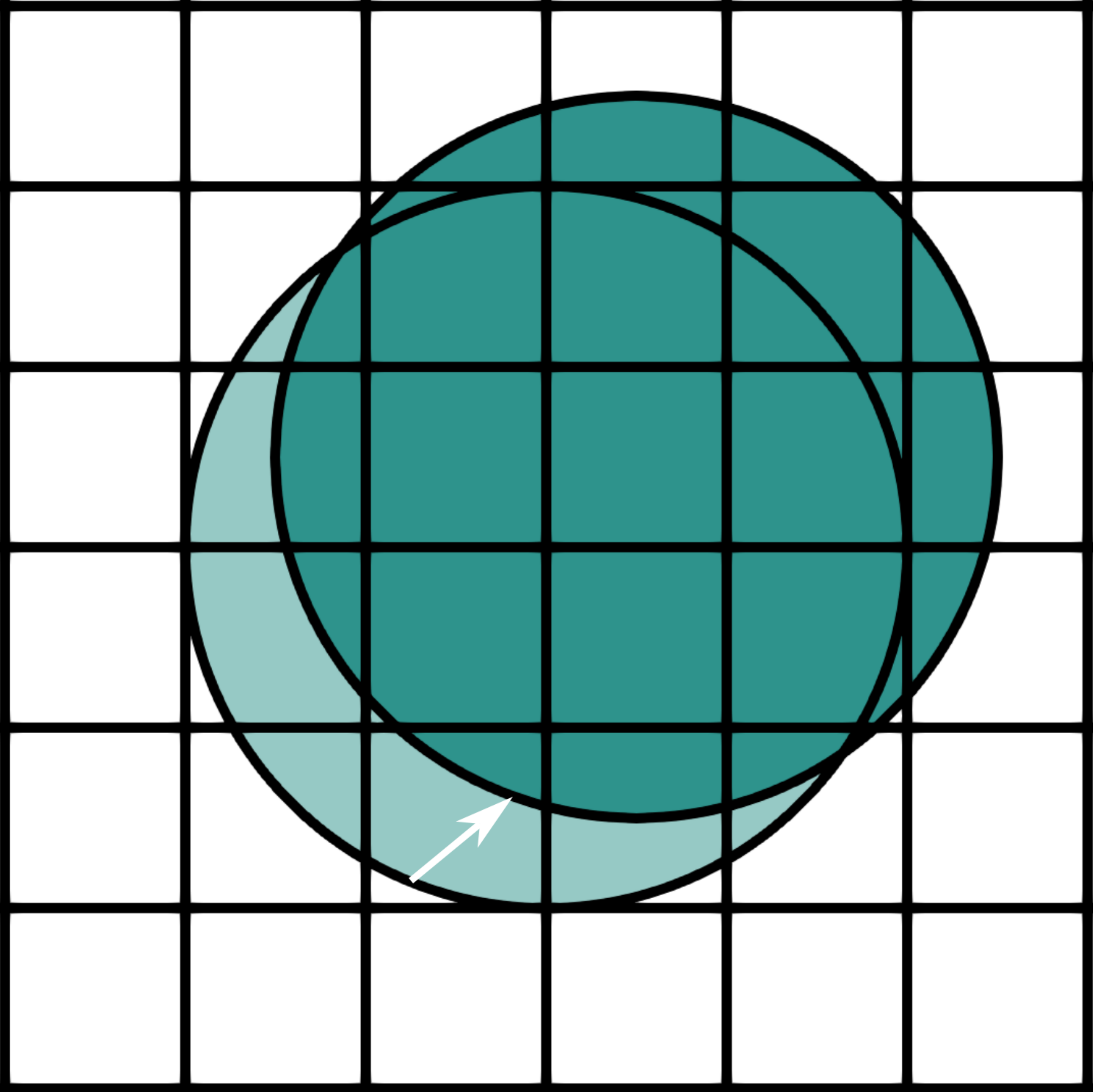}
        \caption{}
    \end{subfigure}
    \hfill
    \begin{subfigure}[t]{0.35\textwidth}
        \includegraphics[width=\textwidth]{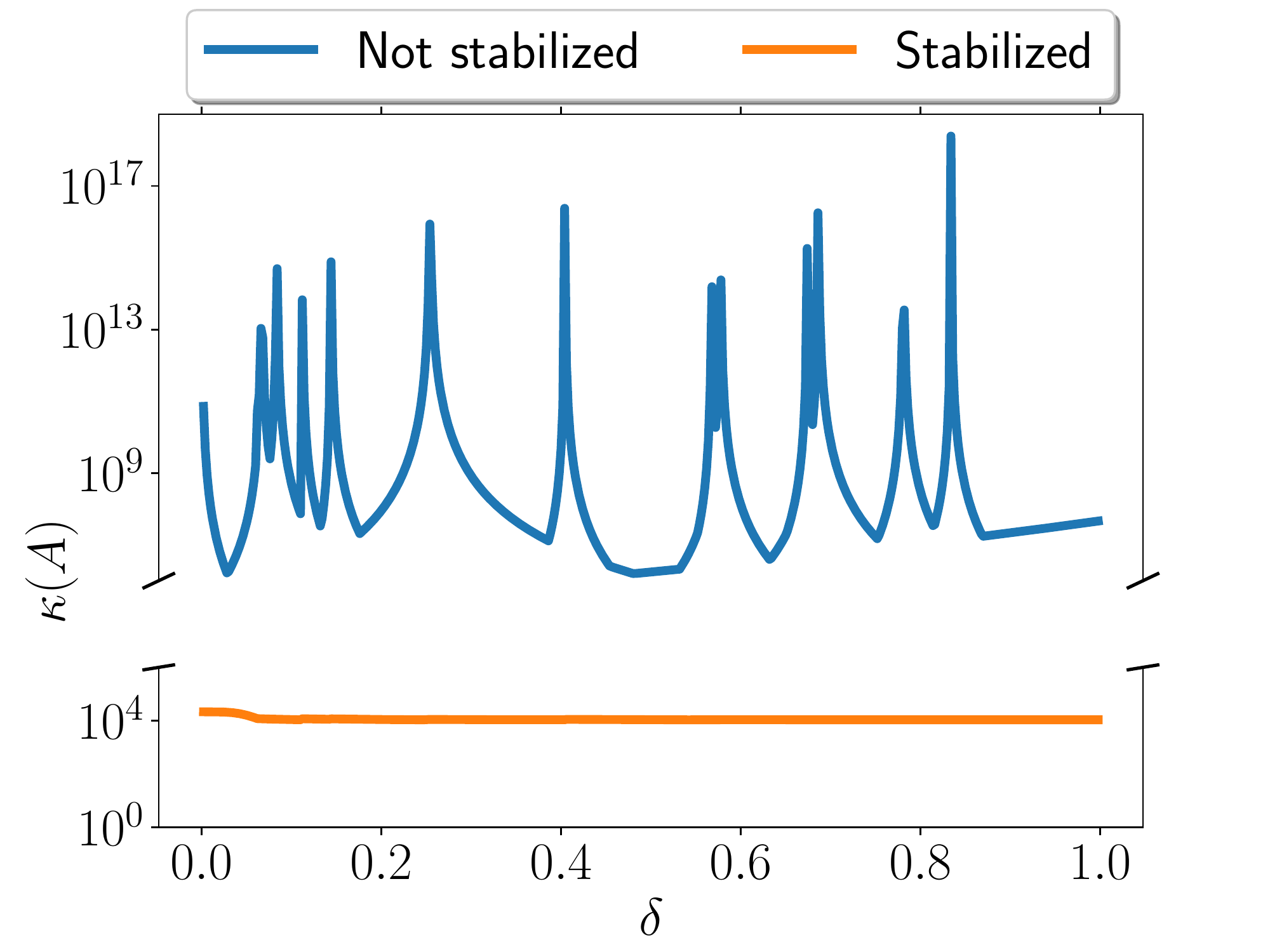}
        \caption{}
    \end{subfigure}
    \hfill
    \begin{subfigure}[t]{0.35\textwidth}
        \includegraphics[width=\textwidth]{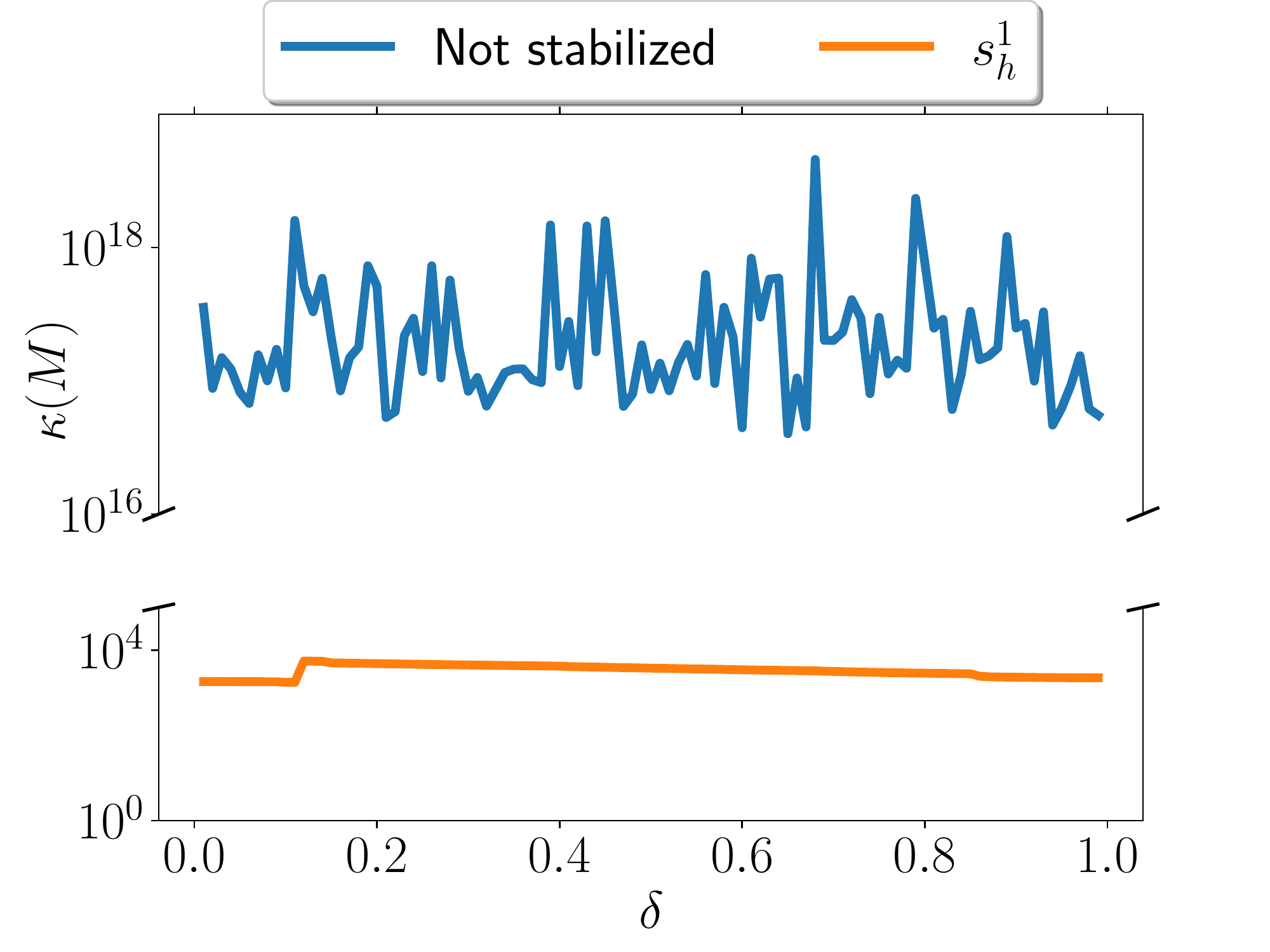}
        \caption{}
    \end{subfigure}
    \caption{(a) Illustration of how the sensitivity analysis is performed, with the position of the intracellular domain for the first and last iteration. (b-c) The condition numbers of the stiffness matrix for the PDE-step (b) and of the mass matrix for the ODE-step (c) versus the translation parameter \(\delta\), without and with two different stabilization formulations. Note that the scaling of the axes is different. For the unstabilized formulation for the ODE-step, \(\delta= 0.315\) and \(\delta=0.43\) give singular matrices.}
    \label{fig:sensitivity}
\end{figure}

\subsection{Geometrical robustness of the unfitted ODE discretization}

To examine the geometrical robustness of the unfitted discretization
of the ODE steps, we again consider a numerical sensitivity study. We
assemble the mass matrix \(\mathcal{M}\) associated
with~\eqref{eq:projection_stable}, on different cut configurations and
with and without stabilization, and compute the corresponding
condition numbers. Specifically, we set \(N=32, M_\delta=100\) and let
the stabilization terms be as defined in \Cref{sec:ode}. We observe
(\Cref{fig:sensitivity} (c)) that the condition numbers for the
unstabilized formulation are very high, and even that the projection
is singular for some cut configurations. The condition numbers of the
$s^1_h$ formulation are substantially lower, on the order of   
$10^3-10^4$. The condition numbers of the $s^2_h$ formulation are
lower yet again, on the order of \(10 - 10^2\).

\subsection{Unfitted ODE discretization converges at optimal order}
\label{sec:numerics:ode}

Next, we conduct numerical experiments to test the convergence rates for the discretization of ODEs on the surface alone, using the following linear system of ODEs defined over the membrane surface:
\begin{equation*}
    v_t  = -s \qquad
    s_t  = v.
\end{equation*}

We consider a structured background mesh over the domain \(\Omega_2=[-1,1]^3\) consisting of uniform cubes, where \(N^3\) is the total number of elements. The level set function defining the surface is defined by
\begin{equation*}
    \varphi_2(x,y,z) = \frac{x^2}{0.8^2}+y^2+\frac{z^2}{0.9^2}-0.8^2.
\end{equation*}
Over this domain, we define manufactured solutions as
\begin{equation*}
    v(x,y,z,t) = (x^2+y^3+z^2)\cos(t), \qquad
    s(x,y,z,t) = (x^2+y^3+z^2)\sin(t).
\end{equation*}
The system is solved on the time interval \([0,2]\) with the number of
uniform time intervals set to \(M=2^{n+1}\) with \(n=0,\dots,4\) and
\(N=4M\). For ease of implementation, we utilize the
$s^1_h$-stabilization here and in the remainder of the manuscript. To
evaluate the discretization properties, we let the discrete
\(L^{\infty}L^2\)- and (square) \(L^2L^2\)-norms be defined by
\begin{align*}
    \norm{v}_{L^{\infty}L^2(\Gamma)} = \max_{m\in [1,\dots,M]} \norm{v(\cdot,t_m)}_{\Gamma}, \qquad
    \norm{v}^2_{L^2L^2(\Gamma)} = \frac{1}{M} \sum^M_{m=1} \norm{v(\cdot,t_m)}^2_{\Gamma}
\end{align*}

By evaluating the errors for the series of meshes and time step sizes
and estimating the errors of convergence (\cref{table:odetest2}), we
observe that the discretization errors decay at first-order, as expected.

\begin{table}
    \begin{center}
        \caption{Errors and convergence rates for the unfitted ODE discretization under spatial and temporal refinement.}
        \label{table:odetest2}
        \begin{tabular}{l*{1}{S[table-format=1.2e-2]}*{1}{S[table-format=-1.2]}*{1}{S[table-format=1.2e-2]}*{1}{S[table-format=-1.2]}}
            \toprule
            {N/M}  & {\(\norm{e_v}_{L^{\infty}L^2(\Gamma)}\)} & {EOC} & {\(\norm{e_v}_{L^2L^2(\Gamma)}\)} & {EOC} \\
            \midrule
            8/2    & 4.64e-01                                 &       & 5.98e-01                          &       \\
            16/4   & 2.14e-01                                 & 1.12  & 2.19e-01                          & 1.45  \\
            32/8   & 9.01e-02                                 & 1.25  & 9.18e-02                          & 1.25  \\
            64/16  & 4.46e-02                                 & 1.02  & 4.14e-02                          & 1.15  \\
            128/32 & 2.43e-02                                 & 0.87  & 1.95e-02                          & 1.08  \\
            \midrule
            {N/M}  & {\(\norm{e_s}_{L^{\infty}L^2(\Gamma)}\)} & {EOC} & {\(\norm{e_s}_{L^2L^2(\Gamma)}\)} & {EOC} \\
            \midrule
            8/2    & 1.07e+00                                 &       & 1.09e+00                          &       \\
            16/4   & 5.94e-01                                 & 0.85  & 5.15e-01                          & 1.08  \\
            32/8   & 2.72e-01                                 & 1.13  & 2.27e-01                          & 1.18  \\
            64/16  & 1.25e-01                                 & 1.12  & 1.05e-01                          & 1.12  \\
            128/32 & 5.98e-02                                 & 1.07  & 4.99e-02                          & 1.07  \\
            \bottomrule
        \end{tabular}
    \end{center}
\end{table}

\subsection{Convergence analysis of the unfitted coupled EMI discretization}

Next, we perform a series of numerical experiments to validate the whole splitting scheme, where we solve the full EMI system~\eqref{emi} augmented by right-hand-sides $f_i, f_e$ in~\eqref{Emi1} and~\eqref{Emi2}, respectively, and a simplified ODE model with given $P_1, P_2$:
\begin{alignat}{2}
    v_t = \frac{1}{C_m}(I_m - P_1 + s), \qquad s_t & = v + P_2 .
\end{alignat}
We consider the geometrical configuration as in Section~\ref{sec:numerics:ode}. We design the manufactured solution,
\begin{align*}
    u_i & = \frac{1}{\sigma_i}  \sin(\pi x) \cos(\pi y) \exp(0.5 z)\exp(-0.5 t), \\
    u_e & = \frac{1}{\sigma_e}  \sin(\pi x) \cos(\pi y) \exp(0.5 z)\exp(-0.5 t), \\
    s   & = I_m,
\end{align*}
and again refine both the spatial and temporal resolution with \(M =[4,6, 8,12,16]\) and \(N=4M\). To measure the errors, we employ the discrete \(L^{\infty}X\)- and square \(L^2X\)-norms defined by
\begin{equation*}
    \norm{u}_{L^{\infty}X(\Omega)} = \max_{m\in [0,\dots,M]} \norm{u(\cdot,t_m)}_{X,\Omega}, \quad
    \norm{u}^2_{L^{2}X(\Omega)} = \frac{1}{M} \sum^M_{m=0} \norm{u(\cdot,t_m)}^2_{X,\Omega},
\end{equation*}
where \(X\) is either the \(L^2\)- or the \(H^1\)- norm. Norms and
rates of convergence are listed in \cref{table:pdetest2_max_single}
for the single-dimensional formulation, and in \cref{table:pdetest2_max_multi} for the multi-dimensional formulation. We observe first-order convergence for \(u\), \(v\), \(s\) and \(I_m\) in all relevant norms, which is as expected in light of the first-order splitting scheme.
\begin{sidewaystable}[htbp!]
    \small
    \begin{center}
        \caption{Errors and convergence rates under spatial and temporal refinement with the single-dimensional formulation.}
        \label{table:pdetest2_max_single}
        \begin{tabular}{l*{1}{S[table-format=1.2e-2]}*{1}{S[table-format=-1.2]}*{1}{S[table-format=1.2e-2]}*{1}{S[table-format=-1.2]}*{1}{S[table-format=1.2e-2]}*{1}{S[table-format=-1.2]}*{1}{S[table-format=1.2e-2]}*{1}{S[table-format=-1.2]}}
            \toprule
            {N/M} & {\(\norm{e_u}_{L^{\infty}L^2(\Omega)}\)} & {EOC} & {\(\norm{e_u}_{L^{\infty}H^1(\Omega)}\)} & {EOC} & {\(\norm{e_v}_{L^{\infty}L^2(\Gamma)}\)} & {EOC} & {\(\norm{e_s}_{L^{\infty}L^2(\Gamma)}\)} & {EOC} \\
            \midrule
            16/4  & 2.94e-02                                 &       & 4.13e-01                                 &       & 7.84e-02                                 &       & 2.51e-01                                 &       \\
            24/6  & 1.89e-02                                 & 1.09  & 2.87e-01                                 & 0.90  & 5.55e-02                                 & 0.85  & 1.45e-01                                 & 1.36  \\
            32/8  & 1.40e-02                                 & 1.05  & 2.20e-01                                 & 0.92  & 4.30e-02                                 & 0.89  & 9.94e-02                                 & 1.30  \\
            48/12 & 9.28e-03                                 & 1.01  & 1.50e-01                                 & 0.95  & 2.98e-02                                 & 0.90  & 6.08e-02                                 & 1.21  \\
            64/16 & 6.96e-03                                 & 1.00  & 1.14e-01                                 & 0.96  & 2.28e-02                                 & 0.94  & 4.36e-02                                 & 1.15  \\
            \bottomrule
        \end{tabular} \\
        \vspace{1em}
        \label{table:pdetest2_l2_single}
        \begin{tabular}{l*{1}{S[table-format=1.2e-2]}*{1}{S[table-format=-1.2]}*{1}{S[table-format=1.2e-2]}*{1}{S[table-format=-1.2]}*{1}{S[table-format=1.2e-2]}*{1}{S[table-format=-1.2]}*{1}{S[table-format=1.2e-2]}*{1}{S[table-format=-1.2]}}
            {N/M} & {\(\norm{e_u}_{L^{2}L^2(\Omega)}\)} & {EOC} & {\(\norm{e_u}_{L^{2}H^1(\Omega)}\)} & {EOC} & {\(\norm{e_v}_{L^{2}L^2(\Gamma)}\)} & {EOC} & {\(\norm{e_s}_{L^{2}L^2(\Gamma)}\)} & {EOC} \\
            \
            16/4  & 2.75e-02                            &       & 3.55e-01                            &       & 6.98e-02                            &       & 2.13e-01                            &       \\
            24/6  & 1.69e-02                            & 1.20  & 2.44e-01                            & 0.93  & 4.82e-02                            & 0.91  & 1.13e-01                            & 1.57  \\
            32/8  & 1.23e-02                            & 1.11  & 1.86e-01                            & 0.94  & 3.71e-02                            & 0.91  & 7.29e-02                            & 1.51  \\
            48/12 & 7.98e-03                            & 1.07  & 1.26e-01                            & 0.96  & 2.54e-02                            & 0.94  & 4.15e-02                            & 1.39  \\
            64/16 & 5.92e-03                            & 1.04  & 9.50e-02                            & 0.97  & 1.93e-02                            & 0.95  & 2.86e-02                            & 1.30  \\
            \bottomrule
        \end{tabular}

        \caption{Errors and convergence rates under spatial and temporal refinement with the multi-dimensional formulation.}
        \label{table:pdetest2_max_multi}
        \begin{tabular}{l*{1}{S[table-format=1.2e-2]}*{1}{S[table-format=-1.2]}*{1}{S[table-format=1.2e-2]}*{1}{S[table-format=-1.2]}*{1}{S[table-format=1.2e-2]}*{1}{S[table-format=-1.2]}*{1}{S[table-format=1.2e-2]}*{1}{S[table-format=-1.2]}*{1}{S[table-format=1.2e-2]}*{1}{S[table-format=-1.2]}}
            {N/M} & {\(\norm{e_u}_{L^{\infty}L^2(\Omega)}\)} & {EOC} & {\(\norm{e_u}_{L^{\infty}H^1(\Omega)}\)} & {EOC} & {\(\norm{e_v}_{L^{\infty}L^2(\Gamma)}\)} & {EOC} & {\(\norm{e_s}_{L^{\infty}L^2(\Gamma)}\)} & {EOC} & {\(\norm{e_{I_m}}_{L^{\infty}L^2(\Gamma)}\)} & {EOC} \\
            \midrule
            16/4  & 2.64e-02                                 &       & 4.20e-01                                 &       & 8.66e-02                                 &       & 2.36e-01                                 &       & 4.46e-01                                     &       \\
            24/6  & 1.92e-02                                 & 0.79  & 2.90e-01                                 & 0.91  & 6.35e-02                                 & 0.76  & 1.08e-01                                 & 1.93  & 3.07e-01                                     & 0.92  \\
            32/8  & 1.30e-02                                 & 1.37  & 2.22e-01                                 & 0.94  & 4.39e-02                                 & 1.29  & 8.08e-02                                 & 1.00  & 2.37e-01                                     & 0.91  \\
            48/12 & 8.48e-03                                 & 1.04  & 1.50e-01                                 & 0.95  & 2.92e-02                                 & 1.01  & 4.85e-02                                 & 1.26  & 1.60e-01                                     & 0.97  \\
            64/16 & 6.23e-03                                 & 1.07  & 1.14e-01                                 & 0.97  & 2.16e-02                                 & 1.05  & 3.35e-02                                 & 1.28  & 1.19e-01                                     & 1.01  \\
            \bottomrule
        \end{tabular} \\
        \vspace{1em}
        \begin{tabular}{l*{1}{S[table-format=1.2e-2]}*{1}{S[table-format=-1.2]}*{1}{S[table-format=1.2e-2]}*{1}{S[table-format=-1.2]}*{1}{S[table-format=1.2e-2]}*{1}{S[table-format=-1.2]}*{1}{S[table-format=1.2e-2]}*{1}{S[table-format=-1.2]}*{1}{S[table-format=1.2e-2]}*{1}{S[table-format=-1.2]}}
            {N/M} & {\(\norm{e_u}_{L^{2}L^2(\Omega)}\)} & {EOC} & {\(\norm{e_u}_{L^{2}H^1(\Omega)}\)} & {EOC} & {\(\norm{e_v}_{L^{2}L^2(\Gamma)}\)} & {EOC} & {\(\norm{e_s}_{L^{2}L^2(\Gamma)}\)} & {EOC} & {\(\norm{e_{I_m}}_{L^{2}L^2(\Gamma)}\)} & {EOC} \\
            \midrule
            16/4  & 2.50e-02                            &       & 3.70e-01                            &       & 7.81e-02                            &       & 2.03e-01                            &       & 4.23e-01                                &       \\
            24/6  & 1.73e-02                            & 0.90  & 2.52e-01                            & 0.95  & 5.59e-02                            & 0.82  & 9.18e-02                            & 1.96  & 2.87e-01                                & 0.96  \\
            32/8  & 1.15e-02                            & 1.43  & 1.89e-01                            & 0.99  & 3.83e-02                            & 1.32  & 7.20e-02                            & 0.84  & 2.16e-01                                & 0.99  \\
            48/12 & 7.36e-03                            & 1.10  & 1.27e-01                            & 0.98  & 2.52e-02                            & 1.03  & 4.07e-02                            & 1.41  & 1.45e-01                                & 0.98  \\
            64/16 & 5.34e-03                            & 1.11  & 9.58e-02                            & 0.99  & 1.85e-02                            & 1.07  & 2.63e-02                            & 1.52  & 1.08e-01                                & 1.03  \\
            \bottomrule
        \end{tabular}
    \end{center}
\end{sidewaystable}

\subsection{Neuronal action potentials and local field potential on unfitted geometries}

We finally consider simulating the neuronal potential and the
surrounding extracellular potential (local field potential) triggered
by a sufficiently large stimulus. The neuron-ECS interface is defined
by a specific neuronal reconstruction\footnote{Data set with ID
    NMO\_76781~\cite{jongbloetsStagespecificFunctionsSemaphorin7A2017}
    from
    \href{neuromorpho.org}{neuromorpho.org}\cite{ascoliNeuroMorphoOrgCentral2007}.}
meshed to an STL format~\cite{morschelGeneratingNeuronGeometries2017}. We define
$\Omega_i$ by the interior of this reconstructed surface, and $\Omega$
by its bounding-box extended by $\pm 8 \si{\micro m}$ in each direction. As
usual, $\Omega_e = \Omega \setminus \Omega_i$.
We use the Hodgkin-Huxley membrane model as described in \Cref{Section:model}, and set all parameters as defined in \cref{table:hh_parameters}. To induce an action potential, we apply a stimulus in the time interval \([0, 0.5]\) in an upper part of the neuron (the part of the membrane which is inside the domain \(D = \{ (x, y, z) \in \Omega_i : y > 100 \si{\micro m}\} \)), illustrated in figure \cref{fig:neuron}. We consider the multi-dimensional formulation of the PDE step, and the $s^1_h$-stabilization of the ODE formulation, and divide the domain into square elements (\(108\times146\times36\)).

The resulting approximations for the membrane, intra and extracellular potentials in three points located close to each other are shown in \cref{fig:neuron}, together with snapshots of the membrane and extracellular potentials. We observe that the cell is depolarized before it is repolarized.

\begin{figure}
    \captionsetup[subfigure]{justification=justified, singlelinecheck=false}
    \begin{subfigure}[t]{0.3\textwidth}
        \centering
        \includegraphics[width=\textwidth]{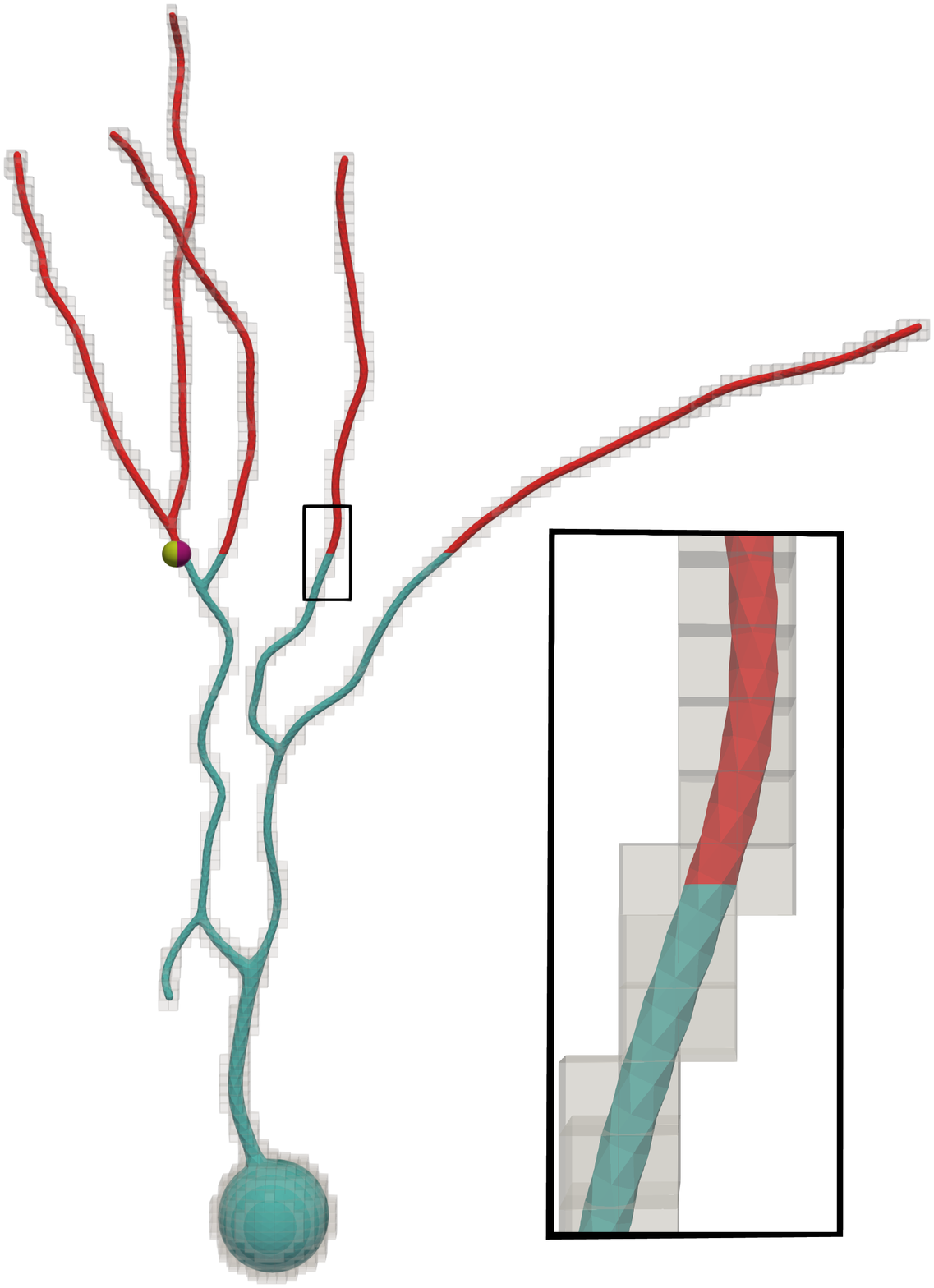}
        \caption{}
    \end{subfigure}
    \hfill
    \begin{subfigure}[t]{0.33\textwidth}
        \centering
        \includegraphics[width=\textwidth]{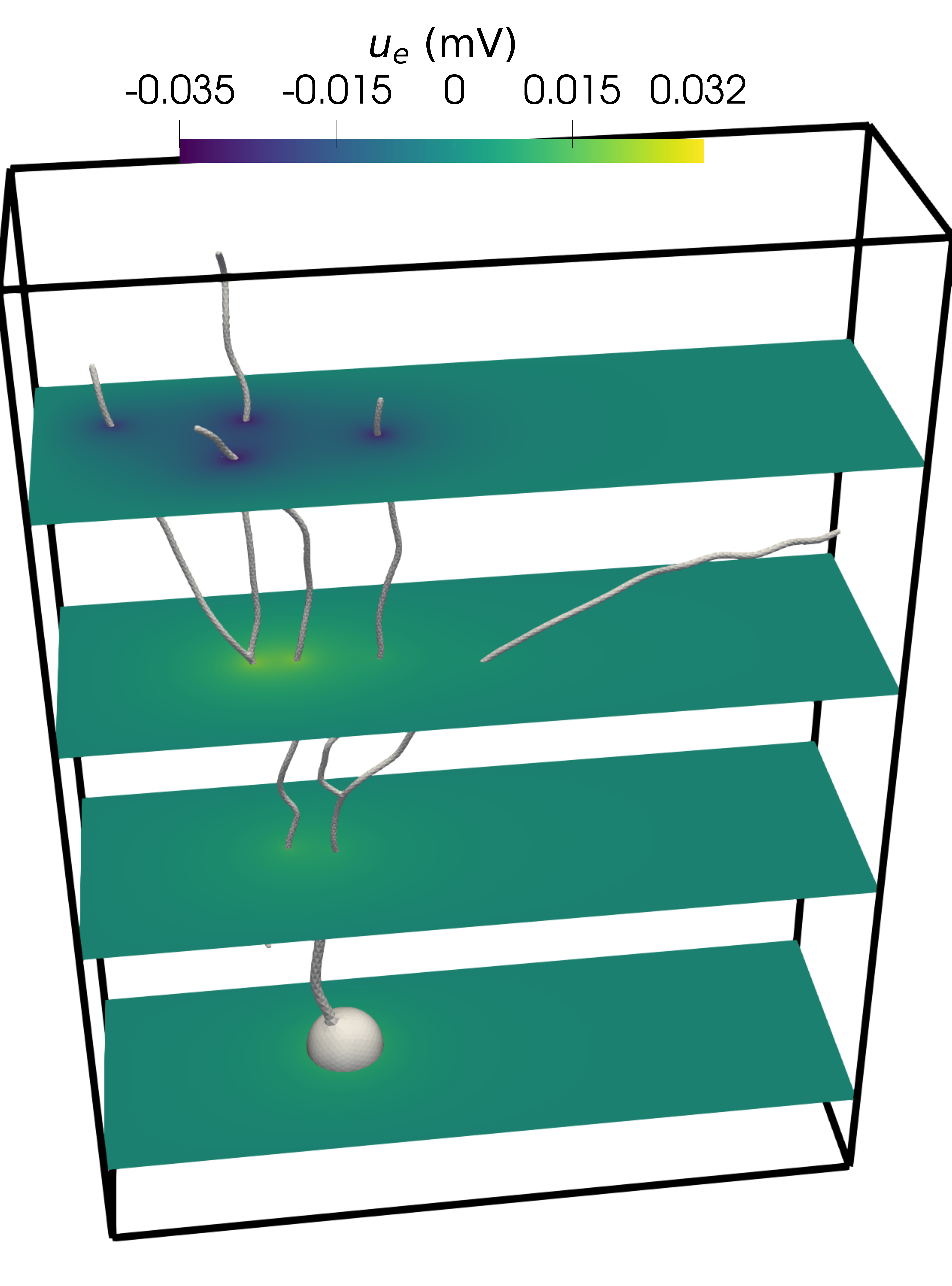}
        \caption[l]{}
    \end{subfigure}
    \hfill
    \begin{subfigure}[t]{0.33\textwidth}
        \centering
        \includegraphics[width=\textwidth]{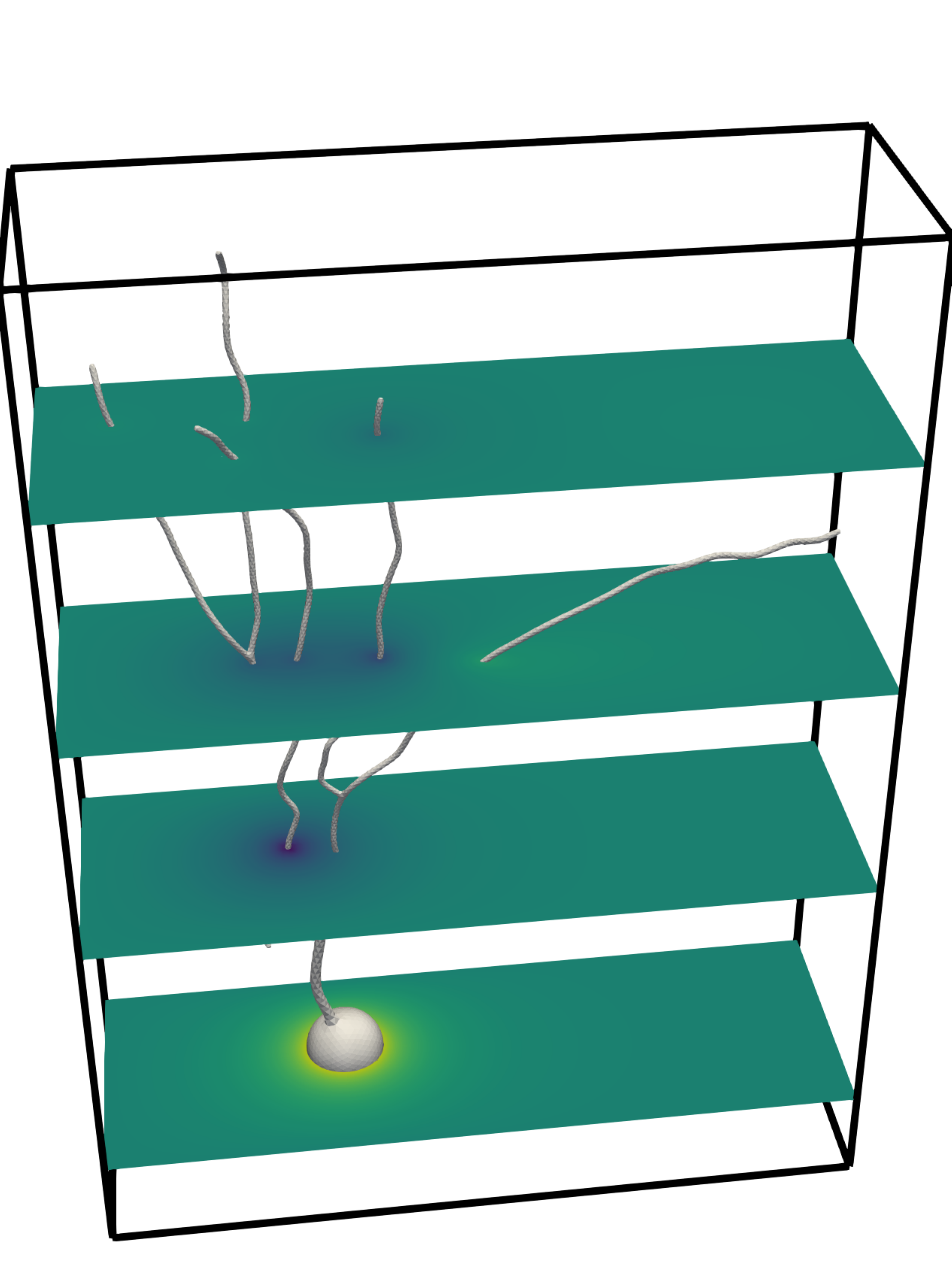}
        \caption[l]{}
    \end{subfigure}
    \begin{subfigure}[t]{\textwidth}
        \centering
        \includegraphics[width=0.24\textwidth]{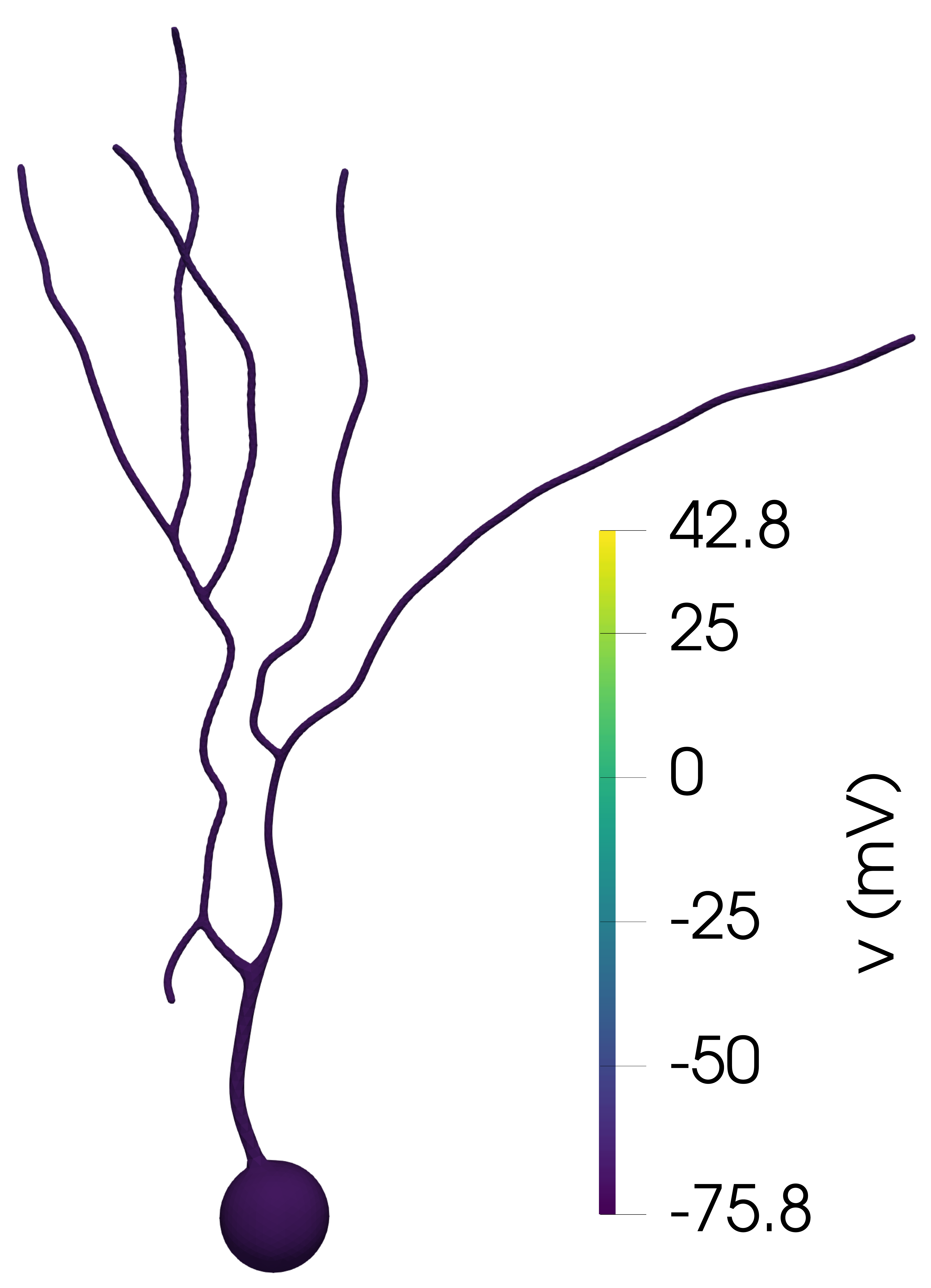}
        \includegraphics[width=0.24\textwidth]{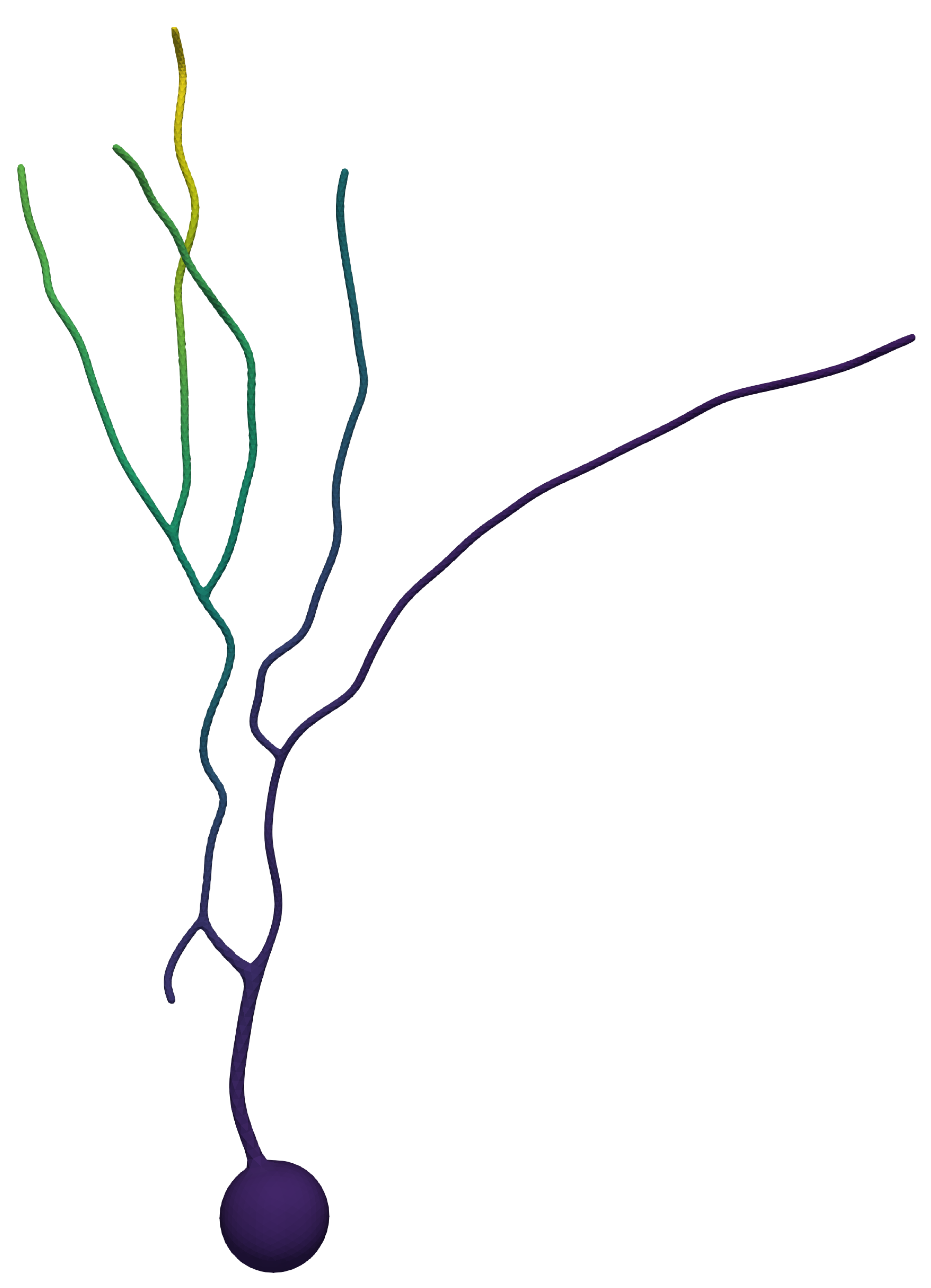}
        \includegraphics[width=0.24\textwidth]{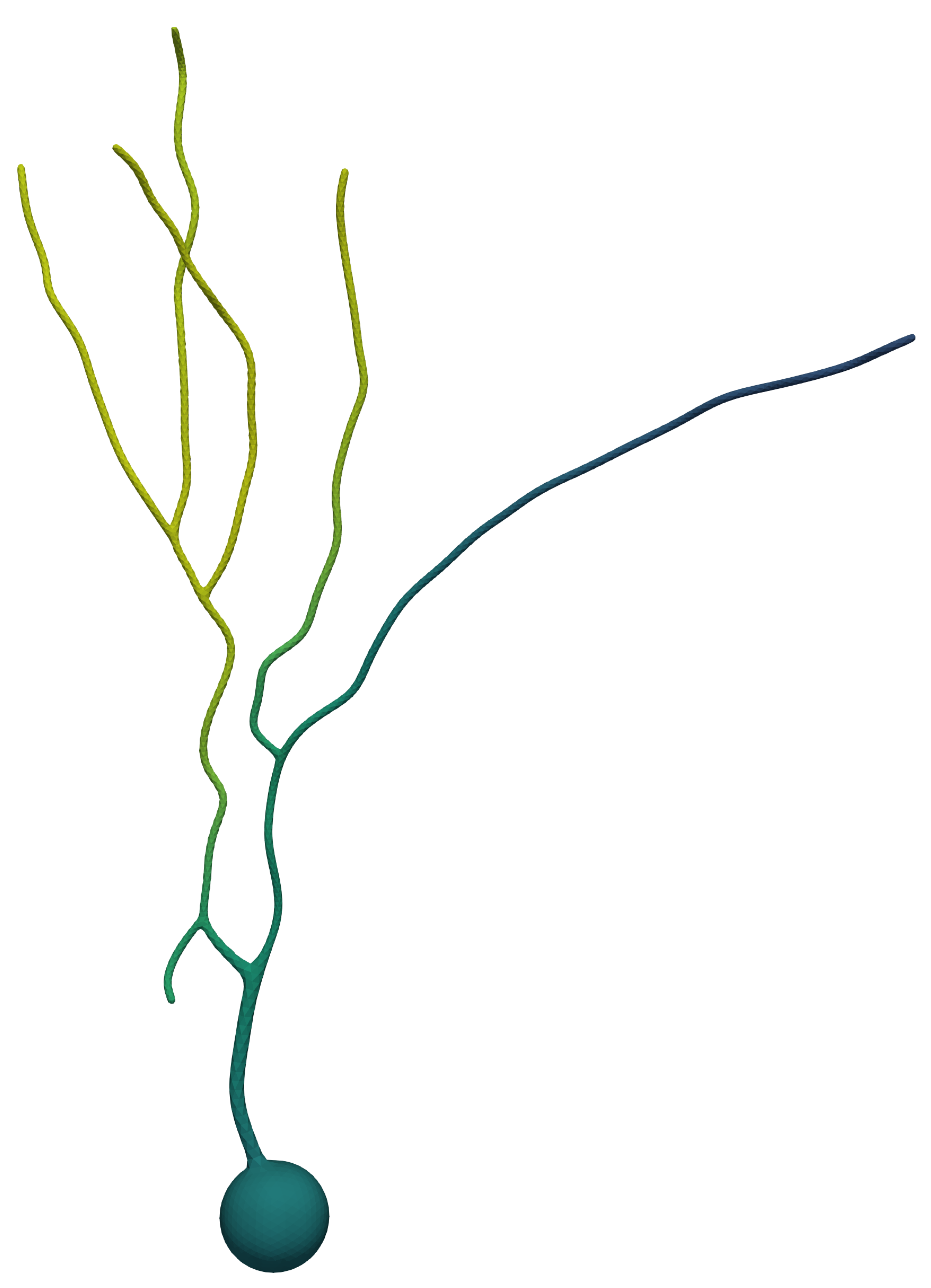}
        \includegraphics[width=0.24\textwidth]{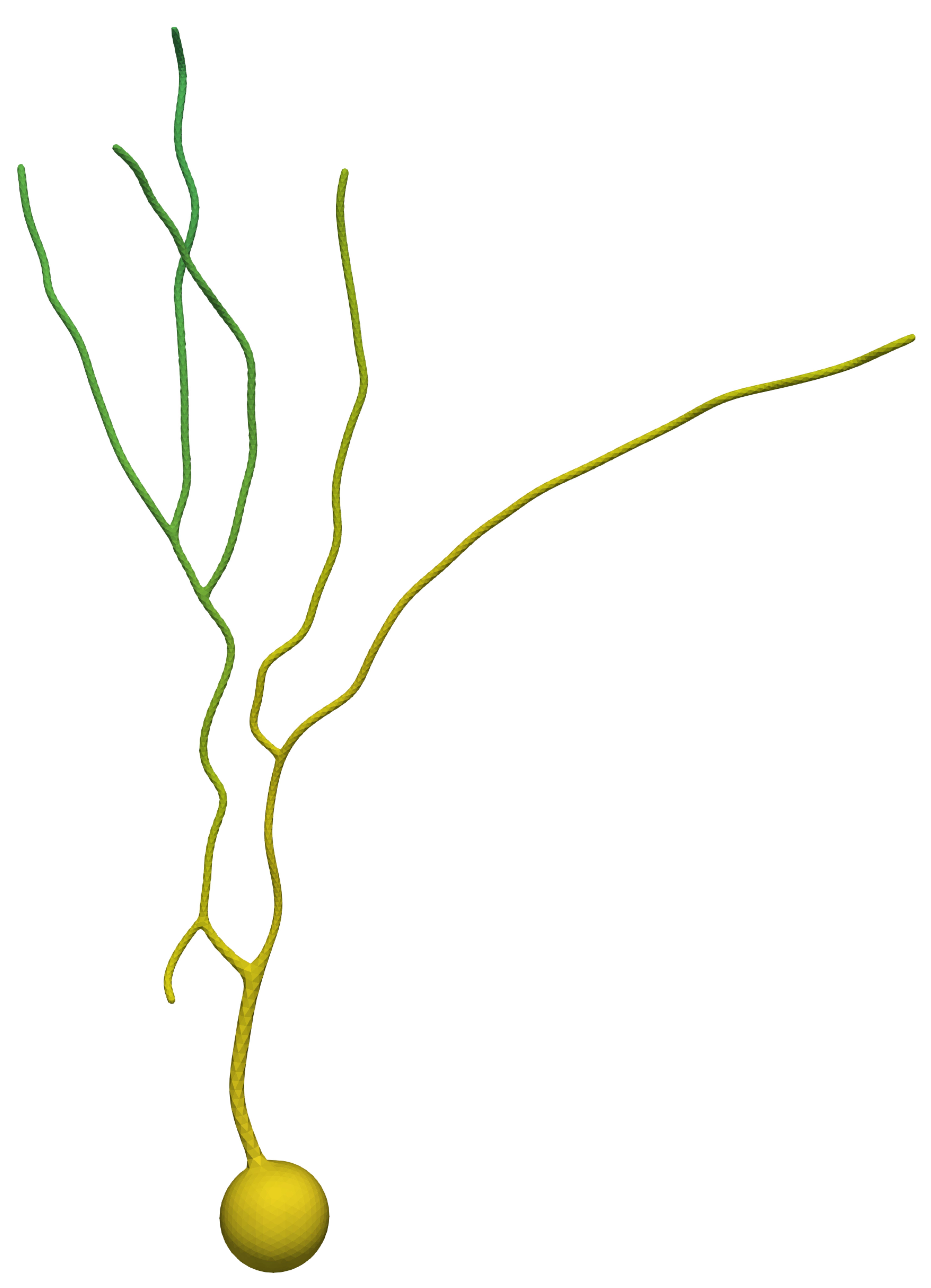}
        \caption{}
    \end{subfigure}
    \begin{subfigure}[t]{0.32\textwidth}
        \includegraphics[width=\textwidth]{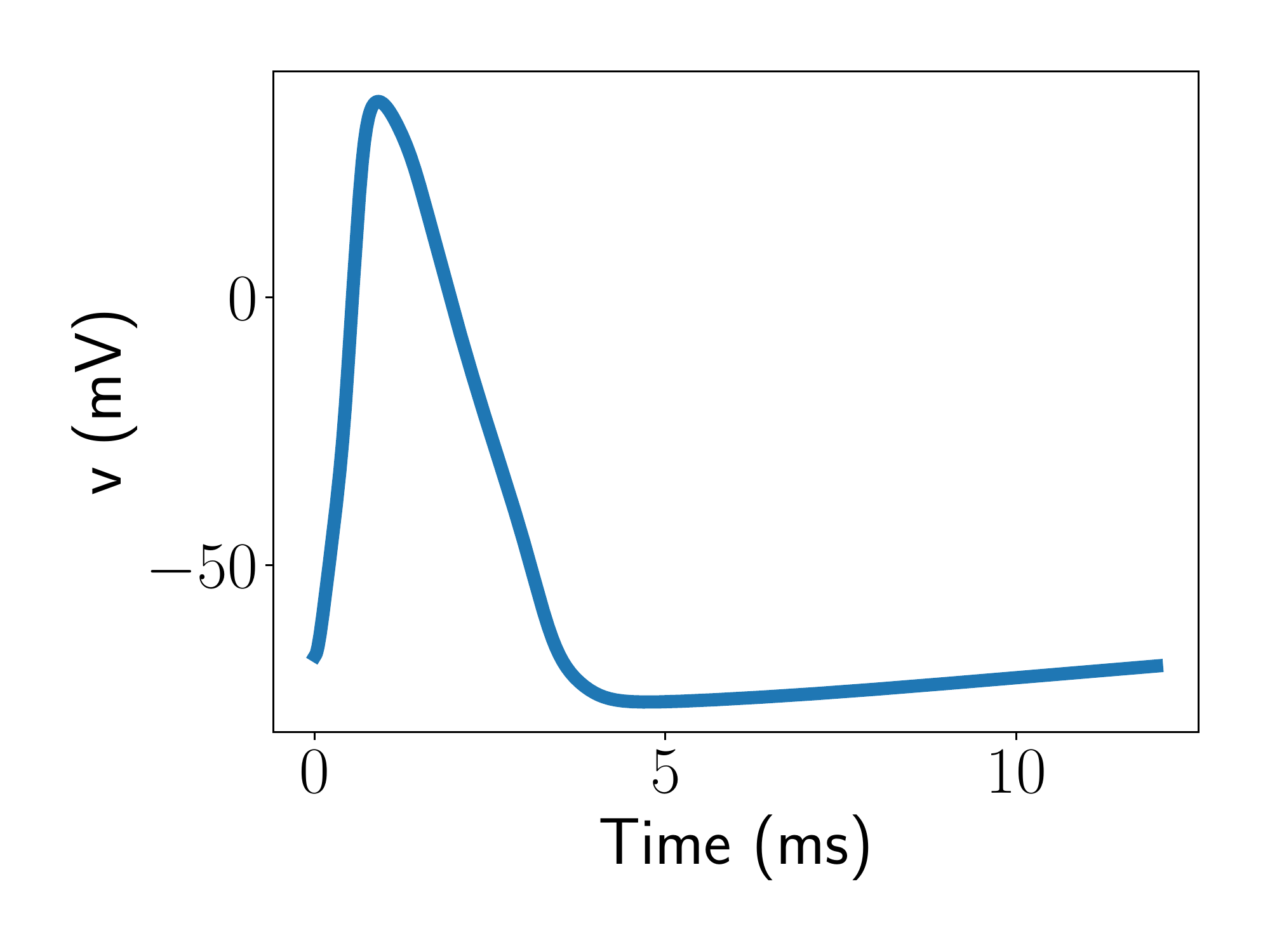}
        \caption{}
    \end{subfigure}
    \hfill
    \begin{subfigure}[t]{0.32\textwidth}
        \includegraphics[width=\textwidth]{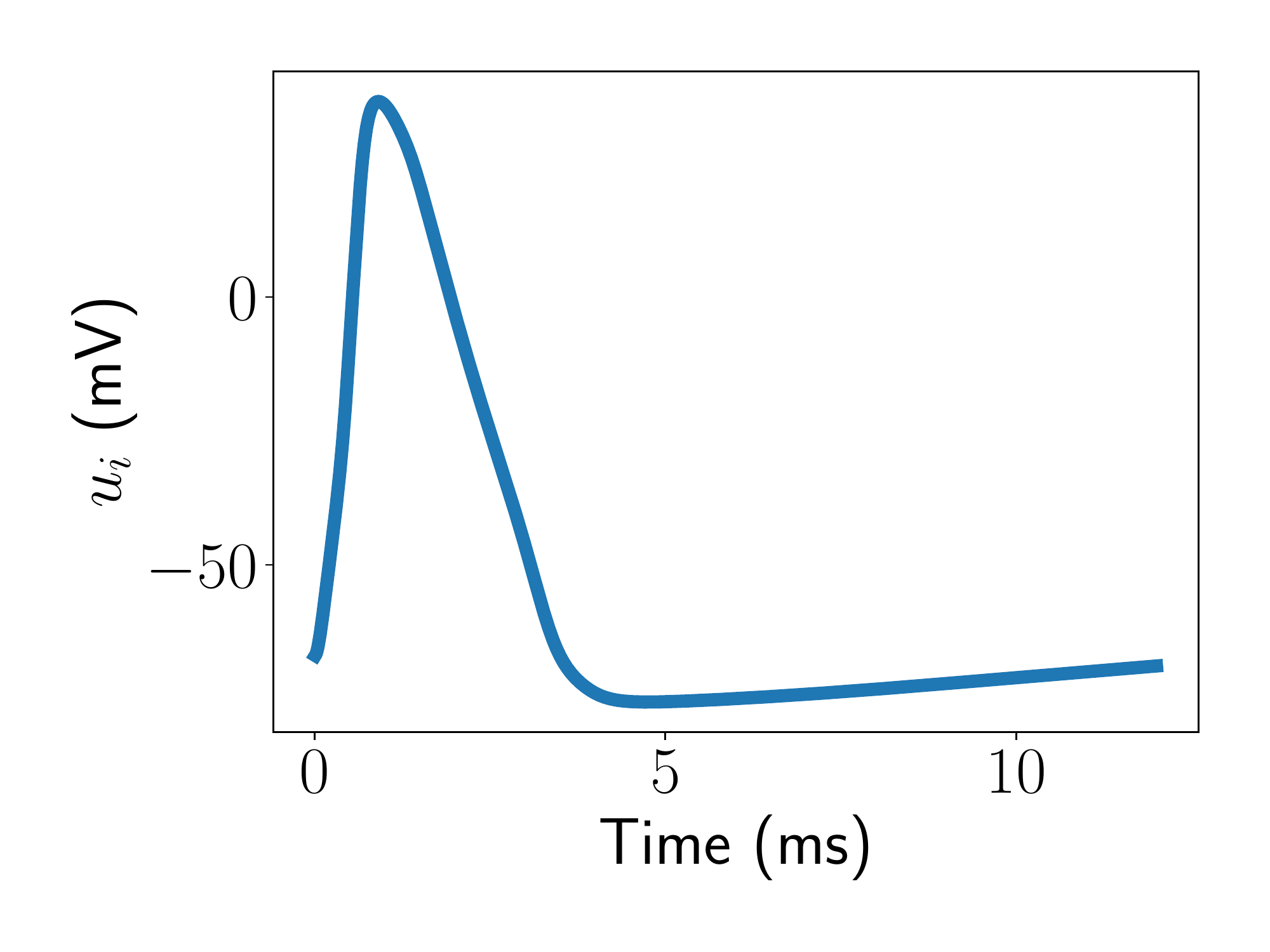}
        \caption{}
    \end{subfigure}
    \hfill
    \begin{subfigure}[t]{0.32\textwidth}
        \includegraphics[width=\textwidth]{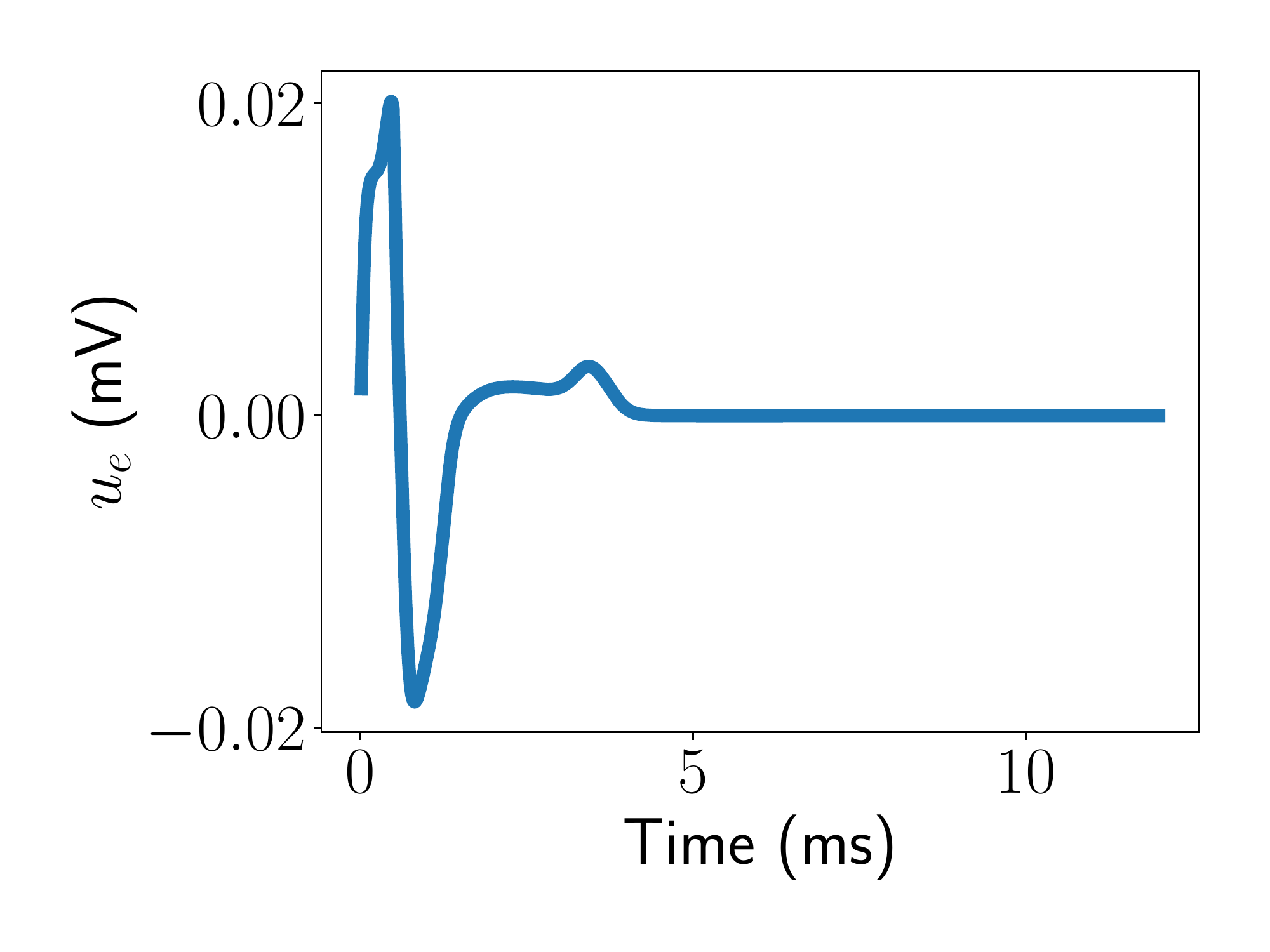}
        \caption{}
    \end{subfigure}
    \begin{subfigure}[t]{0.32\textwidth}
        \centering
        \includegraphics[width=\textwidth]{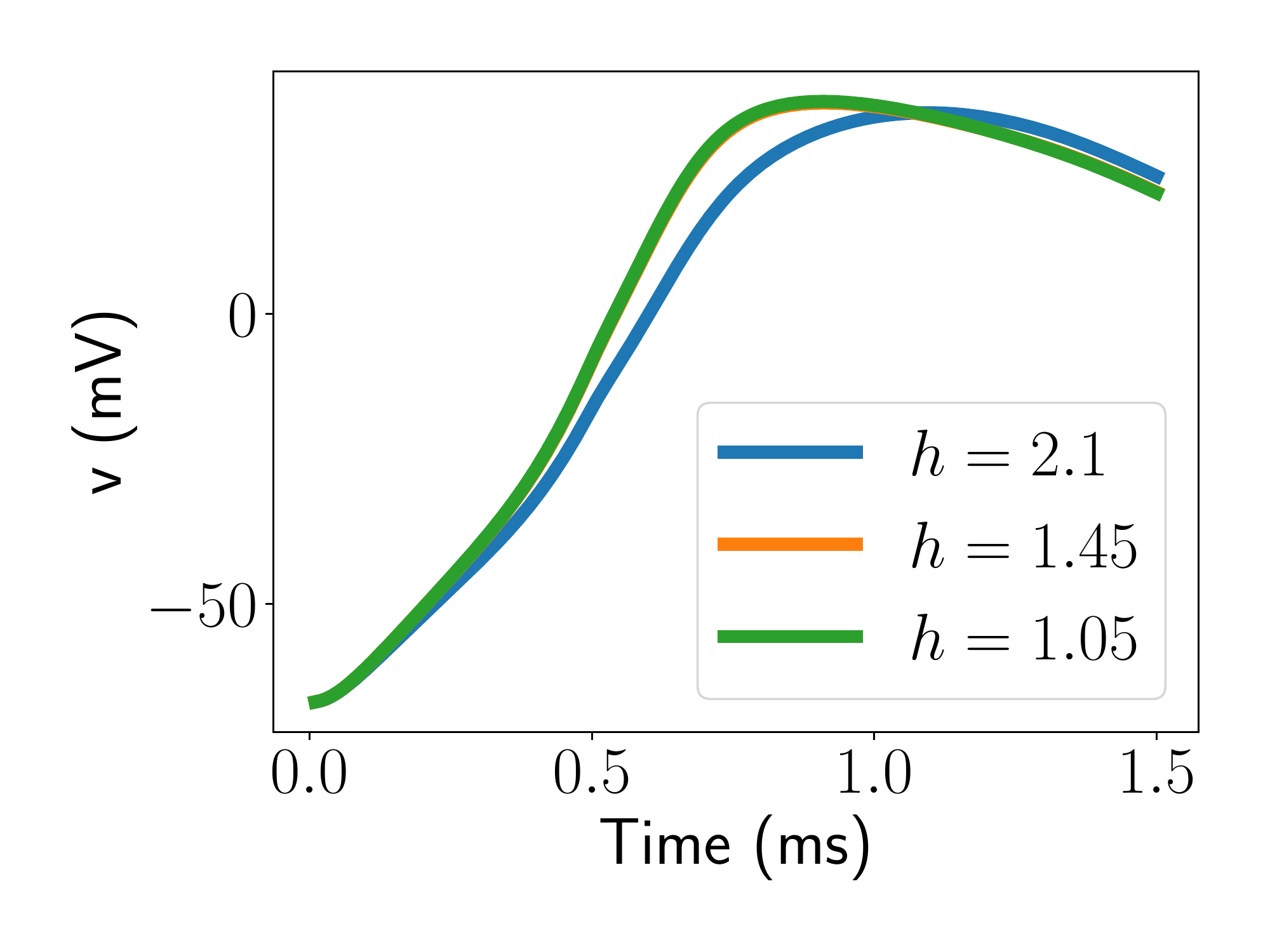}
        \caption{}
    \end{subfigure}
    \hfill
    \begin{subfigure}[t]{0.32\textwidth}
        \centering
        \includegraphics[width=\textwidth]{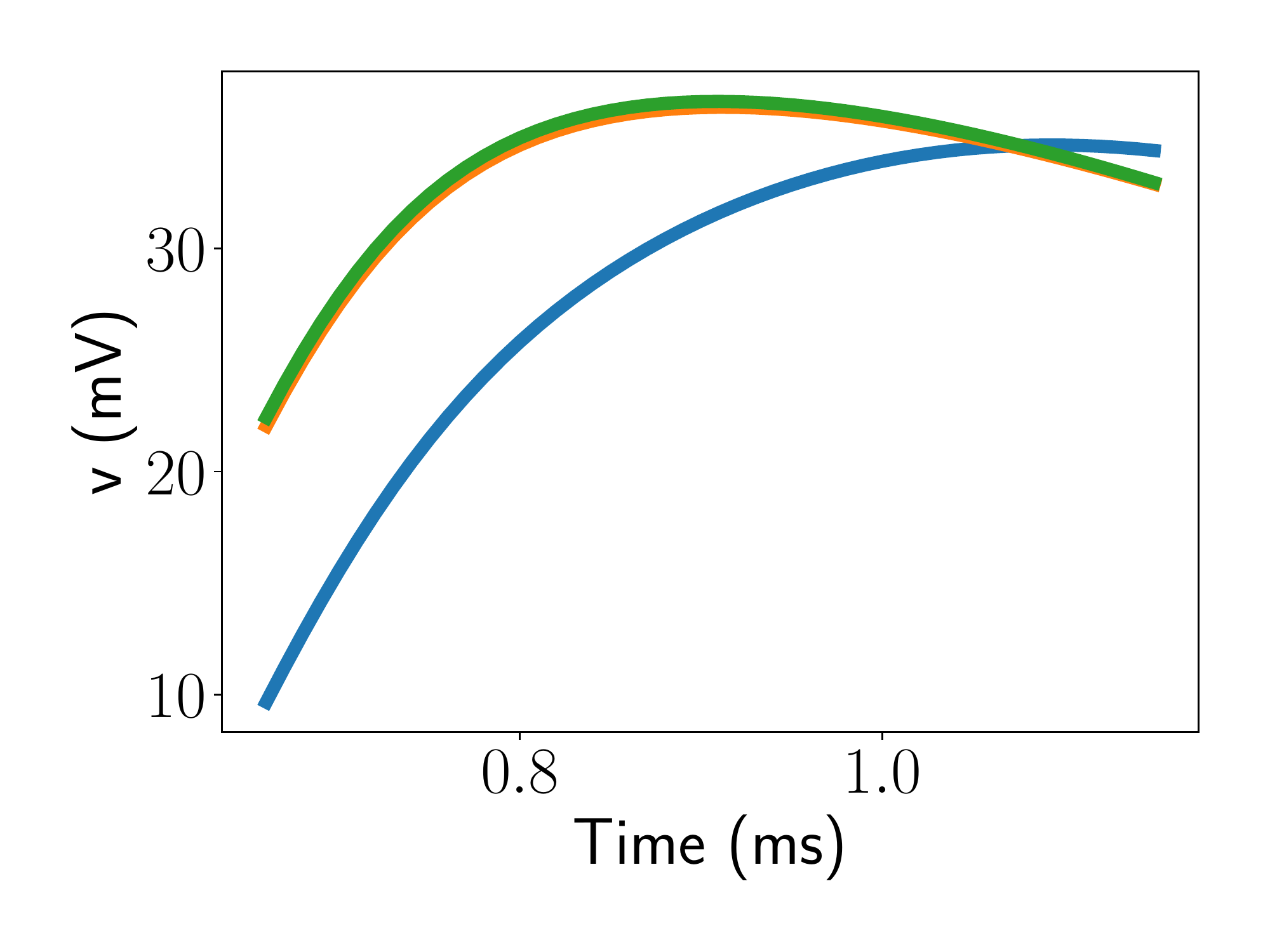}
        \caption{}
    \end{subfigure}
    \hfill
    \begin{subfigure}[t]{0.32\textwidth}
        \centering
        \includegraphics[width=\textwidth]{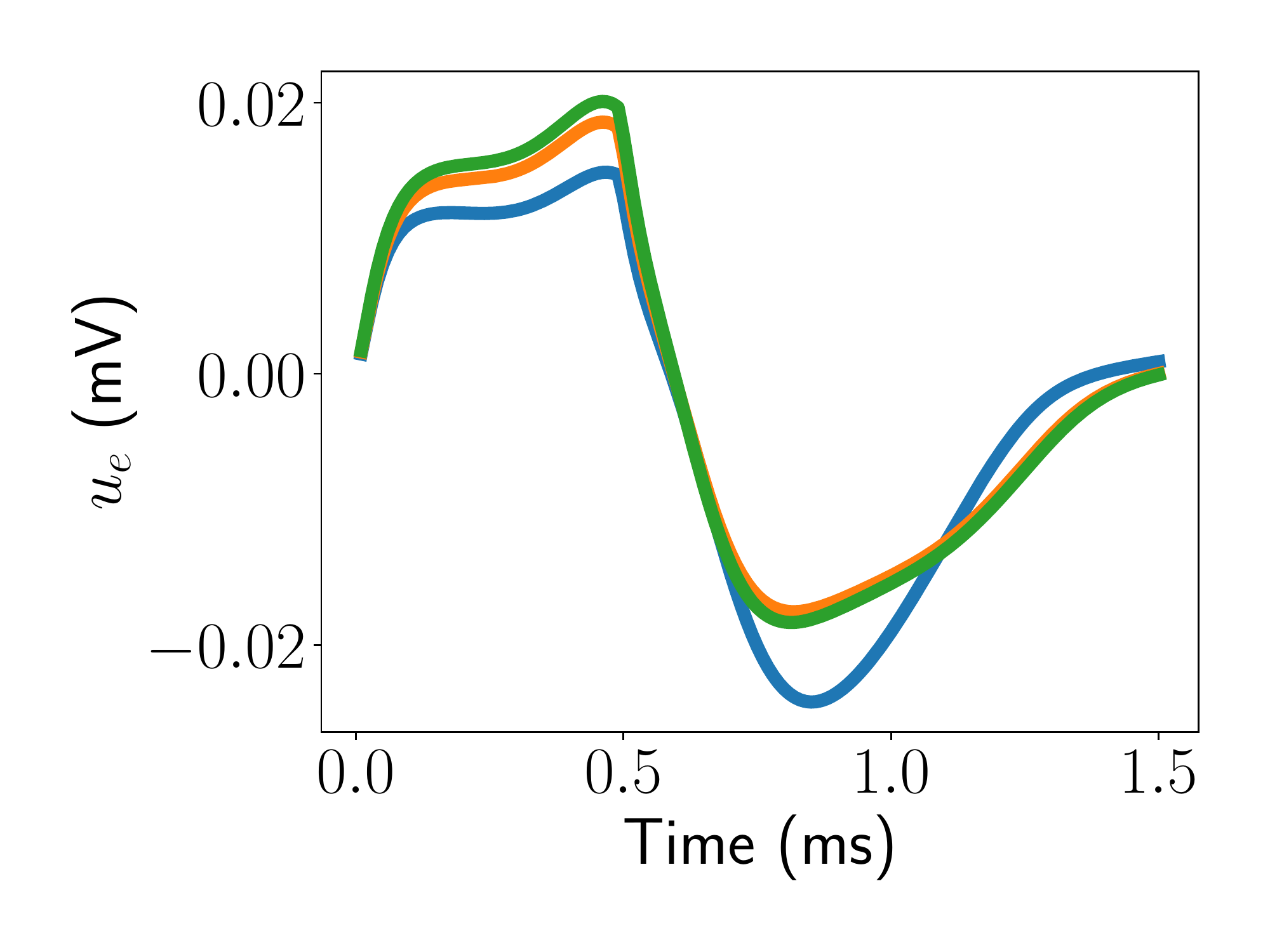}
        \caption{}
    \end{subfigure}

    \caption{ (a) The intracellular domain and its active mesh, showing stimulated parts of dendrites. (b--c) Simulated extracellular potential at
        \(t=0.5\si{ms}\) and \(t=1.1\si{ms}\).
        . (d) Membrane potential at $t = 0.0, 0.5, 1.1, 1.5)
            \si{ms}$. (d) .  (e--g) Traces of the
        membrane (e), intracellular (f) and extracellular (g) potential
        over time, evaluated at close points shown in (a). (h--j) Traces of the neuronal membrane potential (i)
        and (j, zoom) and extracellular potential (h) for different
        spatial resolutions.
    }
    \label{fig:neuron}
\end{figure}

To study the numerical convergence of the membrane and cellular
potential approximations on the unfitted geometry, we also run a
series of spatial refinements (\( 54\times73\times18\),
\(78\times106\times26\), and \( 108\times146\times36\)) on the time interval \([0,1.5 \si{ms}]\)), with
\(\Delta t = 0.01 \si{ms}\). Comparing the membrane and extracellular potentials from the same points as before (\cref{fig:neuron}),  we observe that the solutions seem to converge.

\section{Conclusions and outlook}
\label{sec:conclusion}
We proposed a new CutFEM framework to discretize the EMI problem
on geometrically resolved cell-by-cell models.
Extensive numerical studies showed that our framework gives the same 
convergence rates as expected from a standard fitted finite 
element discretization and that it is insensitive
to how the embedded cell geometry cuts the background mesh 
ensuring the geometrical robustness of our method.
The CutFEM framework for the EMI model circumvents the time-consuming generation
of high-quality unstructured 3D volume meshes required for accurate standard FEM based 
computations. 
To obtain a more holistic simulation pipeline, 
a next natural step would be to implement a user-friendly interface 
which allows to feed imaging data of the neuron cells directly
into the numerical solver.

In this work, we focused on two
CutFEM discretizations of the EMI PDE step both of which were based on a primal
formulation of the potential equations in the extra- and intracellular domains.
Based on the idea of \Hdiv-conforming formulation of the EMI problem as suggested
in, e.g.,~\cite{EMIchapter5}, it will be also interesting 
to explore the advantages and disadvantages of emerging 
\Hdiv-conforming CutFEM formulations~\cite{lehrenfeld2023analysis,frachonDivergencePreservingCut2023} in the context of geometrically resolved
cell-by-cell discretizations.

Solving the EMI model is computationally demanding due to the multiple scales
in space and time characteristic for each subsystem.
Moreover, the linear system arising from the decoupled PDE 
requires solving a large linear system.
To be able to apply the CutFEM framework to large-scale problems
and realistic representation of neural cell networks,
we need to look into the possibilities for parallelization. 
This will require the design and implementation of efficient and parallelizable
mesh intersecting algorithms, and many of the arising challenges 
have been addressed in the work~\cite{KrauseZulian2016}.
Looking into preconditioning to improve
robustness and scalability would also be beneficial. 
Since the problem is coupled across dimensions, this requires tailored methods, 
and for the fitted EMI problem the design of algebraic multigrid methods 
has been considered in the very recent preprint~\cite{budisaAlgebraicMultigridMethods2023}. 
Moreover, to prevent unphysical communication between dendrites through the extracellular space caused by a too low resolution in the background mesh,
we also suggest combining our approach with adaptive mesh refinement based on octrees~\cite{BursteddeWilcoxGhattas2011a}. 
Finally, it might be worthwhile to consider
extending the first-order Godunov splitting scheme to a second-order Strang splitting
to improve computationally efficiency.

\section*{Acknowledgments}
We graciously acknowledge useful discussions and support from Miroslav Kuchta in connection with the neuronal geometry.

\bibliographystyle{siamplain}
\bibliography{reference}

\end{document}